\documentclass[10pt]{amsart}
\usepackage{FancyMath}

\title[Bordered algebras and the wrapped Fukaya category]{Bordered algebras and the wrapped Fukaya category}
\author[Isabella Khan]{Isabella Khan}
\date{10 January 2025}

\begin{document}
	\begin{abstract}
		This paper establishes an isomorphism between endomorphism algebras from the wrapped Fukaya category of a type of punctured surface, and the class of $\A_{\infty}$-algebras related to bordered knot Floer homology, called \emph{star algebras}, which the author first constructed in~\cite{Kh2024}. By viewing the star algebras as $\A_{\infty}$-deformations of underlying associative algebras and making several calculations with Hochschild cohomology, we verify that the star algebras are unique with a given set of generators and basic $\A_{\infty}$-relations. We then make model calculations in order to establish that the endomorphism algebras have these generators and basic operations, so that the desired isomorphism follows.
	\end{abstract}
	\maketitle\blfootnote{The author was partially supported by a NSF Graduate Research Fellowship.}
	\tableofcontents
	
	\thispagestyle{empty}
	
	\section{Introduction}
	
		In order to effectively use knot and link invariants to solve geometric problems, we must first be able to readily compute these invariants in a variety of cases. There are several approaches to such computations. One, which is adopted in the construction of bordered knot Floer homology,~\cite{OzSzab17}, is to define bordered invariants which allow us to compute the original knot invariant as a tensor product of smaller, and often more easily computable, algebraic objects. Knot Floer homology,~\cite{OzSzabKnot},~\cite{Ras}, is a topological invariant of knots $K \subseteq S^3$ defined by choosing a particular Heegaard diagram corresponding to a given projection of $K$. Bordered knot Floer homology is defined by slicing this Heegaard diagram horizontally (that is, along planes perpendicular to the plane of the projection), and assigning algebraic objects to the different portions of the cut diagram in such a way that when we take the tensor product of these objects, we retrieve the knot Floer homology of $K$. Such a construction is possible because the moduli spaces of holomorphic disks, in the cut diagram, which go out to the boundary, are more or less reasonable. In the bordered knot Floer constructions, the behavior of these moduli spaces of disks going out to the boundary is then captured in the algebraic structure of the bordered algebras and modules that arise.

		In more general computational settings, however, the bordered  approach becomes more difficult, because it requires a way to express the more complicated behavior of moduli spaces of objects which go out to the boundary, which is not usually readily available. In these cases, another option is to construct an isomorphism or quasi-isomorphism between two structures, computing the more difficult one in terms of another, more easily computable one. To construct such equivalences between invariants, we first need to know which of the basic algebraic building blocks can be identified, in order to motivate our constructions. These basic identifications are particular interest when one or both of the related algebraic objects has a straightforward combinatorial formulation; by increasing the ease with which the two algebraic objects can be computed, such a combinatorial formulation has the potential to increase the number of possible applications in the future. 
 
 		\begin{wrapfigure}{r}{3.5cm}\vspace{-0.5cm}
			\includegraphics[width = 3.5cm]{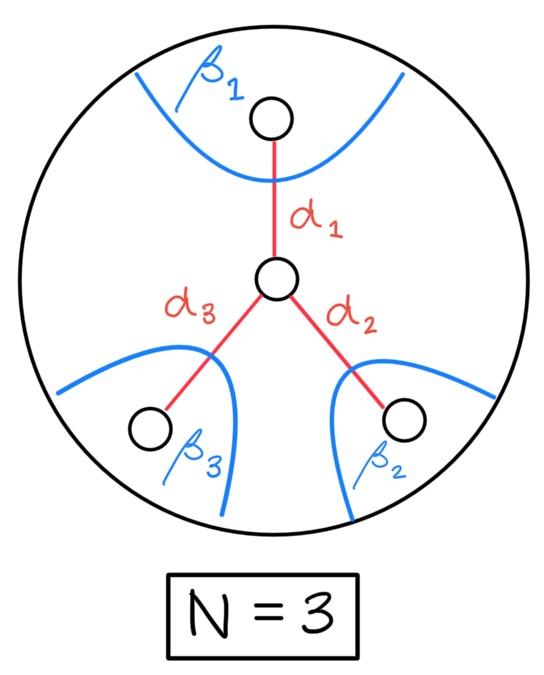} \vspace{-15pt}
	
			\caption{}\label{intro1} \vspace{-.5cm}
		\end{wrapfigure}

		This paper constructs an isomorphism connecting the combinatorially constructed bordered algebras derived from bordered knot Floer homology,~\cite{Kh2024}, and endomorphism algebras associated to the wrapped Fukaya category (see e.g.~\cite{AurFuk10},~\cite{LekPer11},~\cite{Sei08}), of a manifold with boundary associated to the construction of the algebras from~\cite{Kh2024}. In more detail, the bordered algebras in question are a pair of Koszul dual \emph{$\A_{\infty}$-algebras} called \emph{star algebras} (defined in Section~\ref{algs}, below) whose algebraic structure reflected the holomorphic curve theory of a particular class of bordered diagram called a \emph{star diagram}, pictured in Figure~\ref{intro1}. The star diagrams from Figure~\ref{intro1} can also be viewed as symplectic manifolds which are conical near the boundary, in which the red $\alpha$-arcs $\{\alpha_1, \ldots, \alpha_N\}$ and the blue $\beta$-arcs $\{\beta_1, \ldots, \beta_N\}$ can each be viewed as a set of Lagrangians. We can then construct endomorphism algebras having to do with the \emph{wrapped Fukaya category} for this set-up. With this background, the result is as follows. 

		\begin{theorem}\label{fukaya}
			Let $\alpha_1, \ldots, \alpha_N$, $\beta_1, \ldots \beta_N$ denote the $\alpha$- and $\beta$-arcs, respectively. Let $\W(M)$ denote the Fukaya category of the star diagram with $N$ boundary circles, as defined in Section~\ref{fukayaDef}. Let $\A$ and $\B$ be the $\A_{\infty}$-algebras for this star diagram, as constructed in Section~\ref{algs}. Then there exist $\A_{\infty}$-isomorphisms
			\begin{equation}\label{fukaya1}
				\A \cong \End_{\W(M)}(\alpha_1\oplus \cdots \oplus \alpha_N),
			\end{equation}
			and
			\begin{equation}\label{fukaya2}
				\B \cong \End_{\W(M)}(\beta_1 \oplus \cdots \oplus \beta_N).
			\end{equation}
		\end{theorem}

		For example calculations demonstrating how to explicitly express operations from $\A$ and $\B$ in  terms of these endomorphisms, please see Lemmas~\ref{end1},~\ref{end2},~\ref{end3}, and~\ref{end4}, which are model calculations essential to the proof of Theorem~\ref{fukaya}. The figures in the proofs of these lemmas indicate how the result of Theorem~\ref{fukaya} provides another combinatorial (and geometric) way of representing the higher algebra operations. 

		The proof of Theorem~\ref{fukaya} is largely algebraic. First, we review the original construction of the $\A_{\infty}$-algebras $\A$ and $\B$ from~\cite{Kh2024}. The treatment here has two differences from the original construction. First, $\A$ and $\B$ slightly modified from the original algebras used in~\cite{Kh2024}, so that they are over a smaller ground ring, and do not include weighted operations. Second, we view $\A$ and $\B$ as $\A_{\infty}$-deformations of underlying associative algebras $\A_0$ and $\B_0$. In the second half of the paper, we use Lemma~\ref{uni2}, which says that the cobar algebra of $\A_0$ is $\B_0$, and the cobar algebra of $\B_0$ is $\A_0$, to compute the Hochschild cohomologies of the two associative algebras. This allows us to conclude that $\A$ and $\B$ are the unique $\A_{\infty}$-deformations of $\A_0$ and $\B_0$, respectively, with certain basic $\A_{\infty}$-operations (Proposition~\ref{uni1}). We then finish the proof with model calculations (Lemmas~\ref{end1} through~\ref{end4}) which show that the endomorphism algebras can be viewed as as a deformation of $\A_0$ and $\B_0$, respectively, with the desired higher operations.

		This paper should provide the basis for a planned future work which will involve the original Koszul dual \emph{weighted $\A_{\infty}$-algebras} discussed in~\cite{Kh2024}. In particular, the reader will note that the $\A_{\infty}$-algebras $\A$ and $\B$ discussed in this paper are actually stripped down versions of the original weighted algebras constructed and discussed in~\cite{Kh2024}, which include certain higher operations (see Propositions~\ref{alg4} and~\ref{alg8}), but no weighted operations. This is by design. There is, at the moment, no well-defined notion of the endomorphism algebras from Theorem~\ref{fukaya} which could be used for the weighted case; therefore a version of Theorem~\ref{fukaya} for the original Koszul dual weighted $\A_{\infty}$-algebras from~\cite{Kh2024} is not currently known. The author expects, however, that by viewing the original algebras from~\cite{Kh2024} as weighted deformations of the algebras $\A$ and $\B$ discussed in this paper, it will be possible to modify the construction of the endomorphism algebras from Theorem~\ref{fukaya} so as to  analogous equivalence for the original algebras.

		The paper is structured as follows. In Sections~\ref{defs} and~\ref{fukayaDef}, we review the basic definitions and constructions used in this paper, namely $\A_{\infty}$-algebras, Hochschild cohomology, and the wrapped Fukaya category. In Section~\ref{algs}, we review the construction of the underlying associative algebras $\A_0$ and $\B_0$. In Section~\ref{deformations}, we compute the Hochschild cohomology of $\A_0$ and $\B_0$ in order to verify $\A$ and $\B$ are the unique $\A_{\infty}$-deformations of $\A_0$ with the desired basic higher operations. Finally, in Section~\ref{proof}, we use model calculations in the wrapped Fukaya category setting to show that we can apply this uniqueness result to conclude Theorem~\ref{fukaya}.

	\begin{ackno}
		\emph{I would like to thank Peter Ozsv\'ath and Zolt\'an Szab\'o for a great deal of advice and support throughout the preparation of this paper. I would also like to thank Denis Auroux and Robert Lipshitz for many helpful comments, and for pointing out a mistake in an earlier draft of this work.} 
	\end{ackno}

	\section{Algebraic Definitions}\label{defs}
		The aim of this section is to introduce the basic algebraic concepts which will be used in the following sections. None of the definitions or results cited here are original to this paper, and relevant sources are cited in each case.

		\subsection{$\A_{\infty}$-algebras}
			
			Let $R$ be a ring. In this paper, $R$ will usually be $\F_2[V_0, V_{N + 1}]$ for some fixed $N$. An \emph{$\A_{\infty}$-algebra} $\A$ is an $R$-module equipped with maps $\mu_n: \A^{\otimes n} \to \A$ ($n \in \Z_{> 0}$) which satisfy the following conditions: for each $n \geq 0$ and $a_1, \ldots a_n$,
			\begin{equation}\label{def1}
				\sum_{1 \leq r \leq n} \sum_{k = 1}^{n - r + 1} \mu_{n - r + 1} (a_1, \ldots, a_{k - 1}, \mu_r(a_k, \ldots, a_{k + r - 1}), a_{k + r}, \ldots, a_n)  = 0
			\end{equation}
			The combined set of equalities captured by~\eqref{def1} are called the \emph{$\A_{\infty}$-relations for $\A$}. 

			\begin{remark}\label{def2}
				\emph{Note that this definition excludes both $\mu_0$ (curvature) operations, which we are assuming for the purposes of this paper to be identically zero, and any notion of \emph{weighted operations}, as discussed in~\cite{Kh2024} or~\cite{LOT2021}. This is by design: the main result of this paper deals with only the unweighted versions of the algebras from~\cite{Kh2024}, which have neither weighted operations nor non-vanishing $\mu_0$ operations. }
			\end{remark}

			Additionally, given an associative algebra $\A_0$ over $R$, an \emph{$A_n$-deformation} of $\A_0$ is an $\A_{\infty}$-algebra having the same elements and simple multiplication as $\A_0$, but is also equipped with (possibly vanishing) $\mu_1, \mu_3, \mu_4, \ldots, \mu_n$ satisfying the $\A_{\infty}$-relations~\eqref{def1}. 

		\subsection{The cobar complex}\label{CobDef}

		An \emph{augmented associative algebra} is an associative algebra $A$ over a ring $R$ equipped with an augmentation map $\epsilon: A \to R$. Let $A_+ = \ker \epsilon$ be the augmentation ideal and $\Pi: A \to A_+$ the canonical projection. 

		\begin{remark}\label{def3}
			\emph{In practice, $A_+$ will be the complement of the set of idempotent elements of the associative algebra $\A_0$ or $\B_0$, and there is a corresponding unitality condition, which is formulated in Section~\ref{hoch}, below. See the discussion surrounding Proposition~\ref{uni25} for more details.}
		\end{remark}

		Given an augmented associative algebra $A$, define the \emph{reduced bar complex of $A$} as 
		\[
			\bBar(A) = A \otimes A \leftarrow A \otimes A_+ \otimes A \leftarrow \cdots \leftarrow   A \otimes A_+^{\otimes n} \otimes A \leftarrow  \cdots
		\]
		Here the tensor product is taken over $R$, but this is viewed as a chain-complex of $(A,A)$-bimodules with differential $\del: A \otimes A_+^{\otimes n} \otimes A \to A \otimes A_+^{\otimes (n - 1)} \otimes A$ given by
		\[
			\del( a_0 \otimes \cdots \otimes a_{n + 1}) = \sum_{ 0 \leq i \leq n} a_0 \otimes \cdots \otimes a_{i - 1} \otimes a_i a_{i + 1} \otimes a_{i + 2} \otimes \cdots \otimes a_{n + 1}.
		\]
		Let $R_0$ be a finite direct sum of copies of $\F_2$. Usually, we will just have $R_0 = \F_2$. The \emph{reduced cobar complex $\Cob(A)$} is defined to as $\Cob(A) = \bigoplus_{n = 1}^{\infty} (A_+^{\otimes n})^*$, with a differential $\delta^{\Cob}$ which is the transpose of the map
		\[
			\sum_{m, n \geq 0} \id_{A^{\otimes m}} \otimes \mu_2 \otimes \id_{A^{\otimes n}} : \bBar(A) \to \bBar(A).
		\]
		We usually view $\Cob(A)$ as an associative algebra with multiplication given by
		\[
			(f_n \star g_m) (a_1 \otimes \cdots \otimes a_{n + m}) = f_n(a_1 \otimes \cdots \otimes a_n) \cdot g_n(a_{n + 1} \otimes \cdots \otimes a_{n + m}),
		\]
		for $f_n \in (A^{\otimes n})^*, \: g_m \in (A^{\otimes m})^*$. Moreover, when $A$ is finite-dimensional over $\F_2$, we can write $(A^{\otimes n})^* = (A^*)^{\otimes n}$ for each $n$, so that $\delta^{\Cob}$ has the form
		\[
			\delta^{\Cob}(a_1^* \otimes \cdots \otimes a_n^*) = \sum_{i = 1}^{n} a_1^* \otimes \cdots \otimes \mu_2^*(a_i^*) \otimes \cdots \otimes a_n^*,
		\]
		where $\mu_2^*(a^*)(a_1 \otimes a_2) = a^*(a_1 a_2)$, for each $a^* \in A^*$ and $a_1, a_2 \in A$. 

		The cobar algebra is particularly interesting in the context of this paper because of Proposition~\ref{uni2}, which says that each of the two basic augmented associative algebras $\A_0$ and $\B_0$, defined in Section~\ref{algs}, is the cobar of the other. In order to correctly formulate Propositon~\ref{uni2}, however, we must first modify the definition of the cobar algebra by introducing a unitality condition. 
		
		The first step in doing this is to introduce the notion of an \emph{associative algebra with idempotent elements}. Start with an associative algebra $A$ over a ground ring $R$, equipped with \emph{idempotent elements} $\{I_1, \ldots, I_k\}$ with the following properties:
	\begin{enumerate}[label = \textbf{(I\arabic*)}]
		\item $A_+ := A \smallsetminus  \{I_1, \ldots, I_k\}$ is an ideal;

		\item $I_i I_j = \begin{cases} I_i & \text{ if } i = j \\ 0 & \text{ otherwise}\end{cases}$

		\item For each $a \in A_+$, there is are unique $1 \leq i, j \leq k$ such that $I_i a \neq 0$ and $a I_j \neq 0$; these $I_i$ and $I_j$ are called \emph{the initial and final idempotents of $a$}, respectively;

		\item If $I_i, I_j$ are the initial and final idempotents of $a \in A_+$, as above, then $I_i a = a I_j = a$.

		\item~\label{uni23} For $a, b \in A_+$, $ab = 0$ if the initial idempotent of $b$ and the final idempotent of $a$ are not the same. 
	\end{enumerate}
	For an associative algebra $A$ over a ground ring $R$ equipped with such idempotent elements, the augmentation map $\epsilon: A \to R$ used in the definition of $\bBar(A)$ is defined so that $\ker(\epsilon) = A_{+}$. 
	
		\begin{remark}\label{cob72}
			\emph{Note that the set of idempotents has the structure of a commutative ring (with unit $\sum_{i = 1}^k I_i$). In the literature, e.g.~\cite{LOT2021}, the set of idempotents is therefore referred to as the} ring of idempotents, \emph{or} idempotent ring.
		\end{remark}
	
		Now define the modified notion of the cobar complex including the proper unitality condition as follows. Let $A$ be an associative algebra with idempotent elements. Define $\tBar{A}$ to be the subcomplex of the usual (reduced) bar complex $\bBar(A)$ consisting of strings $a_1 \otimes \cdots \otimes a_n \in A_{+}^{\otimes n}$ such that for each $1 < i \leq n$, the initial idempotent of $a_i$ is the same as the final idempotent of $a_{i - 1}$. This is clearly a subcomplex. Define $\tCob(A) = \Hom_{A \otimes A^{\mathrm{op}}}(\tBar(A), A)$ (which is also a subcomplex of the usual cobar $\Cob(A)$).

	\begin{remark}\label{cob71}
		\emph{In~\cite{LOT2021}, the authors omit the extra step of defining $\tBar$ and $\tCob$ and simply work with $\bBar$ and $\Cob$. In their construction, however, the tensor products which make up the  chain complex $\{A \otimes A_+^{\otimes n} \otimes A\}$ for $\bBar$ are implicitly assumed to be over the ring of idempotents. This has the same effect as our extra construction of $\tBar$ and $\tCob$, and the same arguments work. We make this auxiliary construction in order to make it clear that the ground ring is $R$ throughout.}
	\end{remark} 

		\subsection{Hochschild cohomology and $\A_{\infty}$-deformations}\label{hoch}
		
		Let $A$ be an augmented associative algebra with augmentation ideal $A_+$, as in the previous section. The \emph{Hochschild cohomology} of $A$ is defined to be the cohomology of the chain complex
		\[
			 \HC^*(A) = \Hom_{A \otimes A^{\mathrm{op}}} (\bBar(A), A),
		\]
		with differential $\delta: \HC^{n - 1}(A) \to \HC^{n}(A)$ given by
		\[
			(\delta f)(a_0 \otimes \cdots a_{n + 1}) = \sum_{i = 0}^{n} f(a_0 \otimes \cdots \otimes a_i a_{i + 1} \otimes \cdots \otimes a_{n + 1}).
		\]
		
		 In Section~\ref{deformations}, we will compute the Hochschild cohomology of the basic associative algebras $\A_0$ and $\B_0$ and use that to conclude that $\A$ and $\B$ are the unique deformations of $\A_0$ and $\B_0$ given certain properties. In order to be able to argue in this way, we need several facts about Hochschild cohomology. 

		\begin{proposition}\label{hoch1}
			\emph{(Proposition 5.7 of~\cite{LOT 2021})} Let $\A_0$ be an associative algebra and let $\A$ be an $A_{n - 1}$-deformation of $A_0$. Then there is a Hochschild cochain obstruction class $\D_n \in \HC^{n + 1}(\A_0)$ so that the following conditions hold:
			\begin{itemize}
				\item $\D_n$ is a cocycle;
				
				\item $\D_n$ is a coboundary if and only if there is an operation $\mu_n$ making $\A$ into an $A_n$-algebra; indeed the operation $\mu_n \in HC^n(A)$ is a Hochschild cochain with $\delta(\mu_n) = \D_n$;
				
				\item If $\mu_n, \mu_n^{\prime}$ are cochains with $\delta(\mu_n) = \delta(\mu_n^{\prime}) = \D_n$, then $\mu_n - \mu_n^{\prime}$ is a coboundary such that the respective obstruction cocycles $\D_{n + 1}, \D_{n + 1}^{\prime} \in \HC^{n + 2}(A)$ are cohomologous;

				\item If $\A, \A^{\prime}$ are two $A_n$ deformations of $\A_0$ with $\mu_i = \mu_i^{\prime}$ for each $i < n$, and with $\mu_n - \mu_n^{\prime}$ a coboundary, then their respective obstruction cocycles $\D_{n + 1}$ and $\D_{n + 1}^{\prime}$ are cohomologous.
			\end{itemize} 
			
			The statement for maps (which will be used to verify uniqueness of deformations of our associative algebras) is as follows. Given $A_n$-deformations $\A$ and $\A^{\prime}$ of $\A_0$ and an $A_{n - 1}$-homomorphism $f: \A \to \A^{\prime}$ such that $f_1 : \A_0 \to \A_0$ is the identity map, there is an obstruction class $\Ff_n \in \HC^n(\A_0)$ such that
			\begin{itemize}
				\item $\Ff_n$ is a cocycle; 
				
				\item $\Ff_n$ is a coboundary if and only if there is an $A_n$ homomorphism extending $f$;
			\end{itemize}
		\end{proposition}

		\begin{corollary}\label{hoch2}
			\emph{(Corollary 5.22 of~\cite{LOT2021})} Let $\A$ be an $A_n$ structure on the associative algebra $\A_0$. If $\HH^{m+2}(\A_0) = 0$ for each $m \geq n$, then $\A$ extends to an $\A_{\infty}$ structure on $\A_0$; if $\HH^{m+ 1}(\A_0) = 0$ for all $m \geq n$, then this extension is unique up to isomorphism. 
		\end{corollary}

		We next note that we can compute Hochschild cohomology using the cobar complex: the differential on $\Cob(A)$ is given by 
		\begin{equation}\label{hoch5}
			\delta^{\Cob} (a_1^* \otimes \cdots \otimes a_n^*) = \sum_{i = 1}^{n} a_1^* \otimes \cdots \otimes \mu_2^*(a_i^*) \otimes \cdots \otimes a_n^*.
		\end{equation}
		and given this fact, we have:
		\begin{proposition}\label{hoch3}
			\emph{(By induction on the proof of Lemma 5.38 of~\cite{LOT2021})} If $A$ is a filtered algebra with $F_n A$ finite dimensional for each $n$, then there is an isomorphism of chain complexes
			\begin{equation}\label{hoch4}
				A[V_0, \ldots, V_{n}] \otimes \Cob(A) \cong \HC^*(A[V_0, \ldots, V_{n}])
			\end{equation}
			where the differential on the left hand side of~\eqref{hoch4} is given by
			\begin{equation}\label{hoch6}
				\delta(a \otimes \xi) = a \otimes (\delta^{\Cob} \xi)  + \sum_{ \rho \in A : \:\ell(\rho) = 1} (\rho a \otimes (\xi \otimes \rho^*) + a \rho \otimes (\rho^* \otimes \xi)).
			\end{equation} 
			and on the right hand side of~\eqref{hoch4}, the $V_i$ are viewed as elements of the ground ring. 
		\end{proposition}
		
		\begin{remark}\label{uni81}
			\emph{In this paper, we will only be working with the $R = \F_2[V_0, V_{N + 1}]$ because we are not allowing any weighted operations weights (which would force us to use the full ring $\F_2[V_0, \ldots, V_{N+1}]$ as our ground ring). See~\cite{Kh2024} for more details about this.}
		\end{remark}
		
		However, we are only interested in $\A_{\infty}$-deformations of associative algebras with idempotent elements which are unital in the sense that:
		\begin{center}{\small
		\parbox{15cm}{\begin{enumerate}[label = {\textbf{Property~\Roman*}}]
				\item\label{uni24} For each $\mu_j$ with $1 \leq j < n$,
				\[
					\mu_j(a_1, \ldots, a_j) = 0
				\]
				if for any $1 < i \leq j$, the initial idempotent of $a_i$ is not the same as the final idempotent of $a_{i - 1}$.
			\end{enumerate}}}
		\end{center}	
			Given such an algebra, $A$, we therefore define a subcomplex $C^*$ of $\HC^*$, satisfy a unitality condition analogous to the one used to define $\tCob$ at the end of Section~\ref{CobDef}. The subcomplex $C^*$ is defined to be the collection of strings $a_0^* \otimes \cdots \otimes a_{n + 1}^* \in HC^*(A)$ such that the initial idempotent of $a_i$ is the final idempotent of $a_{i -1}$ for each $1 \leq i \leq n + 1$. Note that the differential on the Hochschild cochain complex preserves this property, so this is in fact a bona fide subcomplex.

	We can now modify Lemma~\ref{hoch1} in the following way:
	
	\begin{proposition}\label{uni25}
			Let $A$ be an associative algebra equipped with a set of idempotents, with all notation as in the previous paragraph. Let $\A$ be an $A_{n - 1}$-deformation of $A$, satisfying the unitality condition~\ref{uni24}. Then there is a Hochschild cochain obstruction class $\D_n \in C^{n +1}$ so that the following conditions hold:
			\begin{itemize}
				\item $\D_n$ is a cocycle;
				
				\item $\D_n$ is a coboundary if and only if there is an operation $\mu_n$ satisfying~\ref{uni24} and making $\A$ into an $A_n$-algebra; indeed the operation $\mu_n \in C^n$ is a Hochschild cochain with $\delta(\mu_n) = \D_n$;
				
				\item If $\mu_n, \mu_n^{\prime}$ are cochains with $\delta(\mu_n) = \delta(\mu_n^{\prime}) = \D_n$, then $\mu_n - \mu_n^{\prime}$ is a coboundary such that the respective obstruction cocycles $\D_{n + 1}, \D_{n + 1}^{\prime} \in C^{n + 2}$ are cohomologous;

				\item If $\A, \A^{\prime}$ are two $A_n$ deformations of $\A_0$ with $\mu_i = \mu_i^{\prime}$ for each $i < n$, and with $\mu_n - \mu_n^{\prime}$ a coboundary, then their respective obstruction cocycles $\D_{n + 1}$ and $\D_{n + 1}^{\prime}$ are cohomologous in $C^*$.
			\end{itemize} 
			
			Likewise, the statement for maps is as follows. Start with $A_n$-deformations $\A$ and $\A^{\prime}$ of $\A_0$ such that each of the operations $\mu_i$ on $\A$ and $\mu_i^{\prime}$ on $\A^{\prime}$ satisfy~\ref{uni24}, and an $A_{n - 1}$-homomorphism $f: \A \to \A^{\prime}$ such that $f_1 : \A_0 \to \A_0$ is the identity map and preserving $C^*$ in each degree. Then there is an obstruction class $\Ff_n \in C^n$ such that
			\begin{itemize}
				\item $\Ff_n$ is a cocycle; 
				
				\item $\Ff_n$ is a coboundary if and only if there is an $A_n$ homomorphism extending $f$.
			\end{itemize}
		\end{proposition}
		
		To prove Proposition~\ref{uni25}, one just needs to trace through the proof of Proposition~\ref{hoch1} with the additional assumption that all $\mu_i$ with $1 \leq i < n$ satisfy~\ref{uni24}. It is then clear that under this condition, the $\D_n$ constructed there also satisfies~\ref{uni24}, so when $\D_n$ is a coboundary, the class $\mu_n \in HC^{n}(A)$ with $\delta \mu_n = \D_n$ can be chosen to also satisfy~\ref{uni24}. Likewise, if each of the lower degree maps of $f$ preserve $C^*$, it is clear from the proof of Proposition~\ref{hoch1} that we can choose $\Ff_n$ so that the extended $f$ still preserves $C^*$. 
		
	This then gives the following analogue of Corollary~\ref{hoch2}.
	
	\begin{corollary}\label{uni26}
			Let $A$ be an associative algebra with idempotent elements, as above. Let $\A$ be an $A_n$ structure on $A$. If $H^{r+2}(C^*) = 0$ for each $r \geq n$, then $\A$ extends to an $\A_{\infty}$ structure on $\A_0$; if $H^{r+ 1}(C^*) = 0$ for all $r \geq n$, then this extension is unique up to isomorphism. 
		\end{corollary}

	The last step is to modify Proposition~\ref{uni25} and Corollary~\ref{uni26} to accommodate a $\Z$-valued \emph{Maslov grading} $m: A \to \Z$, in the following way. A map (i.e. Hochschild cochain) which respects the Maslov grading is defined as follows. Let $A$ be an associative algebra and choose $f \in \HC^n(A)$. We say that $f$ \emph{respects the Maslov grading} if and only if there exists a well defined integer $j$ such that for all $a_1, \ldots, a_n \in A$
		\begin{equation}\label{uni13}
			m(f(a_1, \ldots, a_n)) = \sum_{i = 1}^n m(a_i) + i + j - 1.
		\end{equation}
		For such $f$, we define \emph{the Maslov grading of $f$}, $m(f)$, to be equal to $j$. Note that if $f$ respects the Maslov grading, so does $\delta f$. Define $\HC_{\Z}^{*,*}(A)$ to be the subcomplex of Hochschild cochains which respect the Maslov grading. The bigrading is defined in the obvious way, i.e. $f \in \HC_{\Z}^{n, j} (A)$ if and only if $f \in \HC(A)$ respects the Maslov grading and has $m(f) = j$. Likewise, we define $C^{*,*}$ to be the subcomplex of $C^*$ consisting of all cochains in $C^*$ which respect $m$, and $C^{n, j}$ is the collection of $f \in C^n$ respecting $m$ and with $m(f) = j$. We denote the cohomology of $C^{*,*}$ by $H^{*,*}(C^*)$.
	
		\begin{proposition}\label{uni28}
			Let $A$ be an associative algebra with idempotent elements, equipped with Maslov grading $m: A \to \Z$. Let $\A$ be an $A_{n-1}$-deformation of $A$ such that for each $1 \leq j < n$, the operations $\mu_j$ satisfy~\ref{uni24} as well as
			\begin{equation}\label{uni11}
				m(\mu_j(a_1, \ldots, a_j)) = \sum_{i = 1}^j m(a_i) + j - 2,
			\end{equation}
			and operations $f: \TT^*(\A) \to \A$ satisfy
			\begin{equation}\label{uni12}
				m(f_j(a_1, \ldots, a_j)) = \sum_{i = 1}^j m(a_i) + j - 1.
			\end{equation}
			Then the classes $\D_n \in H^{n + 1}(C^*)$ and $\Ff_n \in H^{n}(C^*)$ both respect the Maslov grading, and are in Maslov gradings $-2$ and $-1$ respectively. 
		\end{proposition}
		
		This follows directly from~\eqref{uni13},~\eqref{uni11}, and~\eqref{uni12}. As a result, we have the following condition identifying precisely when we can (uniquely) deform a Maslov graded associative algebra with idempotent elements to get an $\A_{\infty}$-algebra where all operations satisfy~\ref{uni24}.
		\begin{corollary}\label{uni27}
			Let $A$ be an associative algebra with idempotent elements, equipped with Maslov grading $m$. Let $\A$ be an $A_n$-deformation of $A$ equipped with Maslov grading, and such that for each $1 \leq j < n$, the operations $\mu_j$ satisfy~\ref{uni24} as well as~\eqref{uni11} from Proposition~\ref{uni28}. Then:
			\begin{itemize}
				\item \emph{(Existence of deformation)} If $H^{r, -2} (C^*) = 0$ for each $r \geq n + 1$, then $\A$ extends to an $\A_{\infty}$-algebra structure on $A$ such that each $\mu_j$ satisfies~\ref{uni24} and~\eqref{uni11};
				
				\item \emph{(Uniqueness of deformation)} If $H^{r, -1}(C^*) = 0$ for each $m \geq n$, then any two $\A_{\infty}$-algebra structures on $A$ extending $\A$, such that each operation in either structure satisfies~\ref{uni24} and~\eqref{uni11}, are isomorphic as $\A_{\infty}$-algebras.
			\end{itemize}
		\end{corollary}
		This follows from Proposition~\ref{uni25} an Proposition~\ref{uni28} in the same way that Corollary~\ref{hoch2} follows from Proposition~\ref{hoch1}.
		
		Finally, in order to effectively compute the relevant cohomology groups $H^{*,*}(C^*)$ in Section~\ref{HochComp}, we need to modify Proposition~\ref{hoch3} to involve $H^*(C^*)$ instead of the full Hochschild cochain complex. 
		
		 With this modification, the desired result is as follows:
		\begin{proposition}\label{uni29}
			Let $A$ be an associative algebra with idempotent elements, as above. Suppose $A$ is equipped with a filtration such that $F_n A$ is finite dimensional for each $n$. Then there is an isomorphism of chain complexes
			\begin{equation}\label{uni30}
				A[V_0, \ldots, V_{N + 1}] \otimes \tCob(A) \cong C^*[V_0, \ldots, V_{N + 1}]
			\end{equation}
			where the differential on the left hand side of~\eqref{uni30} is given by
			\begin{equation}\label{uni31}
				\delta(a \otimes \xi) = a \otimes (\delta^{\tCob} \xi)  + \sum_{ \rho \in A : \:\ell(\rho) = 1} (\rho a \otimes (\xi \otimes \rho^*) + a \rho \otimes (\rho^* \otimes \xi)).
			\end{equation} 
			and on the right hand side of~\eqref{uni30}, the $V_i$ are viewed as elements of the ground ring. 
		\end{proposition}
		
		This is obvious from the definitions. In Section~\ref{HochComp}, we will use this fact, along with Lemma~\ref{uni2}, which states that each of the algebras $\A_0$ and $\B_0$ is quasi-isomorphic to the cobar of the other, to compute Hochschild cohomology of $\A_0$ and $\B_0$.
		
		Finally, note that the Maslov grading on a tensor product is defined in the obvious way, that is, as
		\begin{equation}\label{uni14}
			m(a \otimes b) = m(a) + m(b).
		\end{equation}
		If $A$ is an associative algebra equipped with Maslov grading $m: A \to \Z$, then the on $\tCob(A)$ is given by
		\begin{equation}\label{uni32}
			m(a^*) = - m(a) - 1
		\end{equation}
		extended to tensor products via~\eqref{uni14}.

	\section{The Wrapped Fukaya category}\label{fukayaDef}
	
		The goal of this section is to present the general definitions and results related to the \emph{wrapped Fukaya category} $\W(M)$ of a particular class of symplectic manifold $(M, \omega)$. None of these are original to this paper -- see e.g.~\cite{LOT2021} and~\cite{AurFuk1} -- and the exposition mostly follows Section 6 of~\cite{LOT2021}. 
	
		\subsection{Definitions}\label{fukaya3}
	
		Let $(M, \omega)$ be a symplectic manifold with conical, cylindrical ends. Let $H$ be a Hamiltonian on $M$, such that near each boundary component of $M$, $H$ is of the form $H(r, x) = r^2$, where $M$ is parametrized as $(0, \infty] \otimes S^1$, near the boundary component. Let $\varphi^1$ denote the time-one flow map for $H$.  Then $\W(M)$ has:
		\begin{itemize}
			\item \textbf{Objects:} Simply connected Lagrangian submanifolds $L$ of $(M, \omega)$, which are conical near the boundary.
			\item \textbf{Homomorphisms:} For simply connected Lagrangians $L_0, L_1 \in \Ob(\W(M))$, 
			\[
				\Hom(L_0, L_1) = CF( \varphi^1(L_0), L_1),
			\] 
			that is, the set of intersections of the (once) wrapped version of $L_0$ with $L_1$; 
			\item \textbf{$\A_{\infty}$-composition map:} The $n$-fold composition map
			\[
				\Hom(L_0, L_1) \otimes \cdots \otimes \Hom(L_{n - 1}, L_n) \to \Hom(L_0, L_n)
			\]
			is an (oriented) count of $(n + 1)$-gons with boundary along
			\[
				\varphi^n (L_0), \varphi^{n - 1} (L_1), \ldots, \varphi^1(L_{n - 1}), \text{ and then } L_0.
			\]
			with the usual coherent, right-handed orientations. For examples of such $(n + 1)$-gons with proper orientation, see the figures surrounding Lemmas~\ref{end1},~\ref{end2},~\ref{end3}, and~\ref{end4}.
		\end{itemize}

		\subsection{Gradings}\label{fukaya2}
		
		The next step is to define a grading for the wrapped Fukaya category of a fixed symplectic manifold $(M, \omega)$. This grading will be constructed in two steps. First, we will construct a groupoid $\G$ using lifts of Lagrangians on $M$, and show that this can be used to define a \emph{groupoid grading} for $\W(M)$. Then, by fixing base-points will define a bona-fide group $G$ from $\G$, and use the $\G$-grading on $\W(M)$ to induce a grading by $G$.
		
		In practice, the group $G$ will be $\Z \times F_{N + 1}$ (where $F_{N + 1}$ is the free group on $(N + 1)$ elements). In Section~\ref{proof}, we will also have $M = S^2 \smallsetminus \{p_0, \ldots, p_{N + 1}\}$, and algebra elements in $\A_0$ and $\B_0$ will correspond to endomorphisms of $\W(M)$. In this case, the $\Z$-component of the grading for a given homomorphism of $\W(M)$ will indicate the Maslov grading of the corresponding algebra element.
		
		Again, the exposition here mostly follows Section 6.2 of~\cite{LOT2021}, and general results are cited without proof.
		
		First, we need to define the notion of a grading for an $\A_{\infty}$-category by a groupoid. Let $\Cc$ be an $\A_{\infty}$ category with composition operations denoted by $\mu_n$, and let $\G$ be a groupoid with central element $\lambda$ and with its operation written as multiplication. Then a \emph{grading $\gr$ of $\Cc$ by $\G$} consists of the following data:
		\begin{itemize}
			\item \textbf{Grading for objects:} For each $L \in \Ob(\Cc)$, a corresponding object $s(L) \in \G$;
			
			\item \textbf{Grading for homomorphisms:} For each pair of $L_0, L_1 \in \Ob(\Cc)$, a decomposition of the set of homomorphisms $L_0 \rightsquigarrow L_1$ as
			\[
				\Hom(L_0, L_1) = \bigoplus_{\gamma \in \Hom_{\G}(s(L_0), s(L_1))} \Hom(L_0, L_1; \gamma);
			\]
			With this notation, define $\gr(x) = \gamma$ for each homogeneous element $x \in \Hom(L_0, L_1; \gamma)$;
		\end{itemize}
		This data is required to satisfy the added property that for each $\mu_n$ and each set of homogeneous elements $x_i \in \Hom(L_i, L_{i + 1})$ ($1 \leq i \leq n$), we have
		\begin{equation}\label{fukDef4}
			\gr(\mu_n(x_1, \ldots, x_n)) = \lambda^{n - 2} \gr(x_1) \cdots \gr(x_n)
		\end{equation}
		Note, in particular, that~\eqref{uni11} is precisely~\eqref{fukDef4} for $\G = \Z$, $\lambda = 1$, and with the group operations written additively. 
		
		Now, the specific groupoid $\G$ by which we will grade $\W(M)$ is defined as follows. Let $\Lag(M) \subseteq TM$ denote the space of Lagrangian submanifolds in $M$. Note that $\Lag(M)$ is a bundle over $M$ with fiber $\lLag(n) \subseteq T_x M = \R^{2n}$ (given $\dim M = 2n$), i.e. the space of Lagrangian subspaces of $\R^{2n}$. In particular, for $M =S^2 \smallsetminus \{p_0, \ldots, p_{n + 1}\}$, the case which arises when we consider the star diagrams, we get $n = 1$, so $\Lag(M)$ is just a circle bundle. In fact, for $M =S^2 \smallsetminus \{p_0, \ldots, p_{n + 1}\}$, $\Lag(M)$ is a trivial circle bundle, and we get
		\begin{equation}\label{fukDef10}
			\Lag(M) \simeq S^1 \times \left( \bigvee_{n + 1} S^1\right),
		\end{equation}
		and choosing a basepoint $\tilde{b} \in \Lag(M)$,
		\begin{equation}\label{fukDef5}
			G := \pi_1 (\Lag(M), \tilde{b}) = \Z \times F_{N + 1}.
		\end{equation}
		
		Now, return to a general symplectic manifold $(M^{2n}, \omega)$. Note that given basepoints $b \in M, \: \tilde{b} \in \Lag(M)$, $\pi_1(M, b)$ acts freely on $\pi_1(\Lag(M), \tilde{b})$, so that $\pi_1(\Lag(M), \tilde{b})$ is a $\Z$-central extension of $\pi_1(M, b)$. We can therefore choose an element $\lambda \in \pi_1(\Lag(M), \tilde{b})$ which projects down to a generator of $\pi_1(\lLag(n))$, an element with Maslov index (in the classical sense) equal to 1. 
		
		Define $\G$ to be the groupoid whose elements consist of all the simply connected Lagrangians of $M$. Morphisms between elements $L_0$ and $L_1$ of $\G$ are defined as follows. For each Lagrangian $L \subseteq (M, \omega)$, there is a canonical lift $\tilde{L}$ of $L$ to the bundle $\Lag(M)$. For $L_0, L_1 \in \G$, let $\tilde{L}_0, \tilde{L}_1$ denote the canonical lifts to $\Lag(M)$. Define $S(L_0, L_1)$ to be the collection of paths $\{[0, 1], \{0\}, \{1\}\} \to \{\Lag(M), \tilde{L}_0, \tilde{L}_1\}$, and define the homomorphisms between $L_0$ and $L_1$ in $\G$ as $\Hom(L_0, L_1) = S(L_0, L_1)$. (We will use $S(L_0, L_1)$ in other contexts, so we want independent notation for it.) Composition 
		\[
			\Hom(L_0, L_1) \times \Hom(L_1, L_2) \to \Hom(L_0, L_2)
		\]
		in $\G$ is defined as $(\gamma_1, \gamma_2) \mapsto \gamma_1 * \eta * \gamma_2$, where $\eta$ is the (canonical, because $L_1$ is simply connected) path in $\tilde{L}_1$ from $\gamma_1(1)$ to $\gamma_2(0)$.
		
		Now, define the grading $\gr$ of $\W(M)$ by $\G$ in the following way. First, since $\Ob(\W(M))$ is the set of (simply connected) Lagrangians on $M$, we just define $s(L) = L$ for each $L \in \Ob(\W(M)) = \G$. Next, we define the grading for homomorphisms of $\W(M)$; that is, given $L_0, L_1 \in \Ob(\W(M))$, and $x \in \varphi^1 (L_0) \cap L_1$, we want to define $\gr(x) \in S(L_0, L_1)$. Define the following intermediate maps:
		\begin{itemize}
			\item $\eta$: For $L \in \Ob(\W(M))$, $m, n \in \Z_{\geq 0}$ $x \in \varphi^m(L)$, let $\eta_x^{\varphi^m(L) \to \varphi^n(L)}$ be the path from $T_x\varphi^m(L)$ to $T{\varphi^{n - m}(x)}\varphi^n(L)$ induced by the Hamiltonian isotopy $H$. In practice, we will usually have $m = 0$ and $n = 1$;
			
			\item $\gamma$: For $L_0, L_1 \in \Ob(\W(M))$ and $x \in L_0 \cap L_1$, let $\gamma_x^{L_0 \to L_1}$ be the path in $\lLag(T_xM)$ from $T_x L_0$ to $T_x L_1$ which is positively oriented with respect to the underlying almost complex structure; when $\dim M = 2$, this will mean that $\gamma_x^{L_0 \to L_1}$ turns $T_x L_0$ clockwise in $\lLag(T_xM)$ until it reaches $T_x L_1$;
		\end{itemize}
		Now, for each $L_0, L_1 \in \Ob(\W(M))$, and $x \in \varphi^1 (L_0) \cap L_1$, we define
		\begin{equation}\label{fukDef6}
			\gr(x) = \eta_{\varphi^{-1}(x)}^{L_0 \to \varphi^1(L_0)} * \gamma_x^{\varphi^{1}(L_0) \to L_1}
		\end{equation}
		If (as we usually do) we want to require that for some fixed $q_0 \in L_0$, and $q_1 \in L_1$, and each $x \in \varphi^1 (L_0) \cap L_1$, $\gr(x) \in S(L_0, L_1)$ starts at $\tilde{q}_0 := T_{q_0} L_0$ and ends at $\tilde{q_1}:= T_{q_1} L_1$, we make the following additional definitions:
		\begin{itemize}
			\item $\epsilon$: Fix $L \in \Ob(\W(M))$ with lift $\tilde{L} \in \Lag(M)$, and $x, y \in L$. Define $\epsilon_{x \to y}^L$ to be path in $\tilde{L}$ from $T_x L$ to $T_y L$, which is unique up to reparametrization.
		\end{itemize}
		We then define
		\begin{equation}\label{fukDef7}
			\gr(x) = \epsilon_{q_0 \to \varphi^{-1}(x)}^{L_0} * \eta_{\varphi^{-1}(x)}^{L_0 \to \varphi^1(L_0)} * \gamma_x^{\varphi^{1}(L_0) \to L_1} * \epsilon_{x \to q_1}^{L_1}
		\end{equation}
		Moreover, $\gr$ is reasonable in the two following senses:
		\begin{proposition}\label{fukDef8}
		\emph{(Proposition 6.9 and Lemma 6.10 from~\cite{LOT2021})}
			\begin{enumerate}[label =(\alph*)]
				\item The $s: \Ob(\W(M)) \to \G$ and $\gr: \Hom_{\W(M)}(L_0, L_1) \to S(L_0, L_1)$ defined above determine a bona-fide grading; in particular, $\gr$ satisfies~\eqref{fukDef4};
				
				\item $\gr$ is well-defined up to Hamiltonian isotopy in the following sense: Let $L_0, L_n$ be two simply connected Lagrangians in $M$, and choose $x \in \varphi^1(L_0) \cap L_n$. Let $x'$ denote the point of $\varphi^n(L_0) \cap L_n$ corresponding to $x$ under dilation by the Hamiltonian $H$. Then, using~\eqref{fukDef7} to write
				\[
					\gr(x) = \epsilon_{q_0 \to \varphi^{-1}(x)}^{L_0} * \eta_{\varphi^{-1}(x)}^{L_0 \to \varphi^1(L_0)} * \gamma_x^{\varphi^{1}(L_0) \to L_n} * \epsilon_{x \to q_1}^{L_n}
				\]
				and
				\[
					\gr(x') = \epsilon_{q_0 \to \varphi^{-1}(x)}^{L_0} * \eta_{\varphi^{-1}(x)}^{L_0 \to \varphi^n(L_0)} * \gamma_{x'}^{\varphi^{1}(L_0) \to L_1} * \epsilon_{x' \to q_1}^{L_1}
				\]
				(using the fact that $\varphi^{-1}(x) = \varphi^{-n}(x')$), we have
				\begin{equation}\label{fukDef9}
					\gr(x) = \gr(x').
				\end{equation}
			\end{enumerate}
		\end{proposition}
		
		This proposition holds in an entirely general setting (provided we only consider simply connected Lagrangians, and that $(M, \omega)$ satisfies the conditions from Section~\ref{fukaya3}). We therefore do not repeat the proof here.
		
		Proposition~\ref{fukDef8} shows that the grading $\gr$ behaves reasonably. It now remains to define a group $G$ related to $\G$, such that $\gr$ induces a $G$ grading on $\W(M)$. As noted above, we define $G$ as
		\[
			G = \pi_1(\Lag(M), \tilde{b}),
		\]
		where $\tilde{b} \in \Lag(M)$ is a fixed basepoint. The $G$-grading is defined as follows. For each simply connected Lagrangian $L \subseteq M$ with canonical lift $\tilde{L} \in \Lag(M)$, choose a basepoint $\tilde{b}_L \in \tilde{L}$. Choose also a fixed anchor path $\eta_L \subseteq \Lag(M)$ for $L$, which goes from $\tilde{b}$ to $\tilde{b}_L$. For each $L_0, L_1 \in \Ob(\W(M))$, we can now make the identification:
		\begin{center}
		\begin{tikzcd}
			f: S(L_0, L_1) \arrow[r, "\cong"] & G
		\end{tikzcd}
		\end{center}
		in the following way: write $\nu_i$ ($i = 0, 1$) for paths in $\tilde{L}_i$ from $\tilde{b}_{L_i}$ to $\gamma(i)$, and then define
		\[
			f(\gamma) = \eta_{L_0} * \nu_0 * \gamma * \overline{\nu_1} * \overline{\eta_{L_1}}
		\]
		where as usual $\overline{\nu_1}, \overline{\eta_{L_1}}$ are the reverse of these paths. So $f(\gamma)$ is:
		\begin{itemize}
			\item $\eta_{L_0}: \tilde{b} \rightsquigarrow \tilde{b}_{L_0}$ then
			\item $\nu_0: \tilde{b}_{L_0} \rightsquigarrow \gamma(0)$ along $\tilde{L}_0$, then
			\item $\gamma: \gamma(0) \rightsquigarrow \gamma(1)$, then
			\item $\overline{\nu_1}: \gamma(1) \rightsquigarrow \tilde{b}_{L_1}$, along $\tilde{L}_1$, then lastly
			\item $\overline{\eta_{L_1}}: \tilde{b}_{L_1} \rightsquigarrow \tilde{b}$.
		\end{itemize}
		Clearly, this defines a bijection between $S(L_0, L_1)$ and $G$. Moreover, since each $L_i$ is simply connected, $f$ is independent of choice of $\nu_0, \nu_1$. Therefore the $\G$-grading $\gr$ on $\W(M)$ induces a $G$-grading, such that for each $x \in \Hom_{\W(M)}(L_0, L_1) = \varphi^1(L_0) \cap L_1$,
		\begin{equation}\label{fukDef11}
			\gr(x) = \eta_{L_0} * \epsilon_{\tilde{b}_{L_0} \to \varphi^{-1}(x)}^{L_0} * \eta_{\varphi^{-1}(x)}^{L_0 \to \varphi^1(L_0)} * \gamma_x^{\varphi^1(L_0) \to L_1} * \epsilon_{x \to \tilde{b}_{L_1}}^{L_1} * \overline{\eta_{L_1}}.
		\end{equation}
		
		Moreover, by Proposition~\ref{fukDef8}(a) and~\eqref{fukDef6}, the $\Z$-component of the grading of~\eqref{fukDef11} satisfies

	\section{Constructions with $\A$ and $\B$}\label{algs}
		The constructions in this section follow those in Sections 6 and 7 of~\cite{Kh2024}, except for two points. First, neither of the $\A_{\infty}$-algebras $\A$ and $\B$ constructed here include \emph{weighted $\A_{\infty}$-operations}. Weighted $\A_{\infty}$-operations will not be discussed further in this paper, except to note how the proofs of $\A_{\infty}$-relations for $\A$ and $\B$ can be adapted from the ones given in~\cite{Kh2024}; therefore we do not define them here. The second key difference is exposition that in constructing $\A$ and $\B$, we first define associative algebras $\A_0$ and $\B_0$, and then construct $\A$ and $\B$ as deformations of $\A_{0}$ and $\B_0$, respectively. 

For the remainder of the paper, fix $N > 2$, and a single ground ring $R = \F_2[V_0, V_{N + 1}]$.

		\subsection{The $\A_{\infty}$-algebra $\A$}\label{alphaAlg}

		First, define $\A_0$ to be an associative $R$-algebra with the following generators:
		\begin{itemize}
			\item $I_i$, $i = 1, \ldots, N$

			\item $U_i$, $i = 1, \ldots, N$;
	
			\item $s_i$, $i = 1, \ldots, N$
		\end{itemize}
		
		$\A_0$ is equipped with a \emph{Maslov grading} $m$ and an \emph{Alexander grading} $A$, which are defined as follows. The Maslov grading $m$ is a map $\A_0 \to \Z$ defined by:
		\begin{align*}
			m(a) &= 0 \text{ for all } a \in \A \\
			m(V_0) &= 2N - 2 \\
			m(V_{N + 1}) &= 2
		\end{align*}
		The Alexander grading is a \emph{homological grading}; that is, for a choice of $2N$ distinct points $p_1, \ldots p_{2N}$ arranged in counterclockwise orientation in $S^1$, $A$ is a map $\A_0 \to H_0(S^1 \smallsetminus \{p_1, \ldots, p_{2N}\}; \Z) \cong\Z^{2N}$. Label the generators of $H_0(S^1 \smallsetminus \{p_1, \ldots, p_{2N} \}; \Z)$ as $\h_1, \ldots, \h_N$. Then define
		\begin{align*}
			A(I_j) = A(U_j) &= \h_{2j - 1} \text{ for each } 1 \leq j \leq N \\
			A(s_j) &= \h_{2j} \text{ for each } 1 \leq j \leq N \\			A(V_0) &= \sum_{j = 1}^{2N} \h_j \\
			A(V_{N + 1}) &= \sum_{j  = 1}^{N} \h_{2j}
		\end{align*}

		With generators as above, we define $\mu_2$ in the following way (writing the the $\mu_2$ as common multiplication):
		\begin{align*}
			I_i I_j &= \begin{cases} I_i & \text{ if } i = j \\ 0 & \text{ otherwwise} \end{cases} \\
			I_j U_i = U_i I_j &= \begin{cases} U_i & \text{ if } i = j \\ 0 & \text{ otherwise} \end{cases} \\
			I_j s_i &= \begin{cases} s_i & \text{ if } i = j \\ 0 & \text{ otherwise} \end{cases} \\
			s_i I_j &= \begin{cases} s_i & \text{ if } i = j - 1 \\ 0 & \text{ otherwise} \end{cases} \\
			U_i U_j &= 0 \text{ unless } i = j \\
			U_i s_j &= s_j U_i = 0 \text{ for all } i, \: j \\
			s_i s_j &= 0 \text{ unless } i = j - 1; 
		\end{align*}
		The $I_j$ are called \emph{idempotents}. Fix $1 \leq i < j \leq N$. Define
		\begin{equation}\label{alg1}
			s_{ij} = s_i s_{i + 1} \cdots s_{j - 1}
		\end{equation}
		Define $s_{N1} = s_N$. In the notation of~\eqref{alg1}, $s_i = s_{i(i+1)}$, for each $1 \leq i \leq N$, provided $(i +1)$ is counted $\emm N$. For $1 \leq j \leq i \leq N$, define
		\begin{equation}\label{alg2}
			s_{ij} = s_{iN} \cdot s_{N} \cdot s_{1j}
		\end{equation}
		Alternately, definitions~\eqref{alg1} and~\eqref{alg2} may be summarized by saying that for each $1 \leq i, j \leq N$, 
		\begin{equation}\label{alg3}
			s_{ij} := s_i s_{i  +1} \cdots s_{j - 1},
		\end{equation}
		where all indices are counted $\emm N$. However, definitions~\eqref{alg1} and~\eqref{alg2} are, perhaps, clearer.
		
		Additionally, define
		\[
			U_4 = \sum_{i = 1}^N s_{ii},
		\]
		so $U_4$ has one component starting / ending in each idempotent. Now $\{I_{i}\}_{i = 1}^N$ is an idempotent set in the sense of Section~\ref{CobDef}, and we define $\A_+ = \A_0 \smallsetminus \{I_i\}_{i = 1}^N$, and the augmentation map $\epsilon: \A_0 \to R$ to have $\ker \epsilon = \A_+$.

		Next, fix a non-zero element $a \in \A_0$. Because $a$ is a product of $I_j^{\prime}$s, $U_j^{\prime}$s, and $s_j^{\prime}$s, there exists unique $1 \leq j_1, j_2 \leq N$ such that 
		\begin{align*}
			I_{j_1} a &\neq 0, \\
			a I_{j_2} &\neq 0.
		\end{align*}
		Define the \emph{initial idempotent of $a$} to be $I_{j_1}$, and define the \emph{final idempotent of $a$} to be $I_{j_2}$. Often, by abuse of notation, we will simply refer to $j_1$ and $j_2$, respectively, as the initial and final idempotents of $a$.

		Next we define $\A$ as a deformation of $\A_0$; that is, as an $\A_{\infty}$-algebra which is equal to $\A_0$ as an associative $R$-algebra (and is equipped with the same Maslov and Alexander grading as $\A_0$) but now has higher operations.  Higher operations on $\A$ are defined in Sections 4.2 and 4.3 of~\cite{Kh2024}. Non-zero higher operations are defined to be in one-to-one correspondence with \emph{allowable unweighted planar graphs}. However, we will not use these graphs further in this paper, so we simply summarize the key results about $\A_{\infty}$-operations on $\A$ below. 
		
		To do this, we need a few definitions. For $a \in A$, suppose 
		\[
			A(a) = \sum_{j = 1}^{2N} k_j \h_j.		\]
		Then define the \emph{length}, $\ell$, of $a$ as: 
		\begin{equation}\label{alg61}
			\ell(a) = \sum_{j = 1}^{2N} k_j.
		\end{equation}
		Note that for each $a \in A$, $\ell(a) > 0$, and that $\ell(a) = 0$ if and only if $a$ is one of the basic elements $U_i, s_i$. Next, fix a sequence $(a_1, \ldots, a_n) \in \A$, $n > 2$. Define
		\begin{equation}\label{alg62}
			A(a_1, \ldots, a_n) = \sum_{i = 1}^n A(a_i).
		\end{equation}
		This sequence is defined to be \emph{centered} 
		\begin{equation}\label{alg6}
			\sum_{i = 1}^n \ell(a_i) = j \cdot 2N,
		\end{equation}
		for some $j \in \N$, \emph{left-extended} if we can write $a_1 = a_1^{\prime} a_1^{\prime \prime}$ such that $(a_1^{\prime \prime}, a_2, \ldots, a_n)$ is a centered sequence, and \emph{right-extended} if we can write $a_n = a_n^{\prime \prime}a_n^{\prime}$ such that $(a_1, \ldots, a_{n - 1}, a_n^{\prime \prime})$ is a centered sequence. Note that when $(a_1, \ldots, a_n)$ is left-extended, the decomposition $a_1 = a_1^{\prime} a_1^{\prime \prime}$ is unique with the property that $(a_1^{\prime \prime}, a_2, \ldots, a_n)$ is centered. Likewise when $(a_1, \ldots, a_n)$ is right-extended, the decomposition $a_n = a_n^{\prime \prime} a_n^{\prime}$ is unique with the property that $(a_1, \ldots, a_{n - 1}, a_n^{\prime prime})$ is centered.

		We are now ready to discuss higher operations. The following lemma is adapted from Theorem 4.1 of~\cite{Kh2024}:
		\begin{lemma}\label{alg4}
			Fix $a_1, \ldots, a_n \in \A$, for $n \geq 2$. Then $\mu_n(a_1, \ldots, a_n) \neq 0$ if and only if the following conditions hold:
			\begin{enumerate}[label = (\roman*)]
				\item \emph{(The idempotents match up)} The initial idempotent of $a_i$ is the final idempotent of $a_{i - 1}$ for each $1 \leq i \leq n$;

				\item \emph{(The Maslov grading works out)} $n = j(2N - 2) + 2$, for $j \in \N$;

				\item \emph{(Extension)} The sequence $(a_1, \ldots, a_n)$ is either centered, or it is left- or right-extended, but not both;
				
				\item \emph{(The Alexander grading works out)} The condition varies depending on whether the sequence $(a_1, \ldots, a_n)$ is centered, left- or right-extended.
				\begin{itemize}
					\item	If centered: $A(a_1, \ldots, a_n) = j \cdot \sum_{i = 1}^{2N} \h_i$;
					
					\item If left-extended: Write $a_1 = a_1^{\prime} a_1^{\prime \prime}$ such that $(a_1^{\prime \prime}, a_2, \ldots, a_n)$ is centered. Then $A(a_1, \ldots, a_n) = A(a_1^{\prime}) + j \cdot \sum_{i = 1}^{2N} \h_i$ (for the same $j$ as in (ii));

					\item If right-extended: Write $a_n = a_n^{\prime \prime} a_n^{\prime}$ such that $(a_1, \ldots, a_{n - 1}, a_n^{\prime \prime})$ is centered. Then $A(a_1, \ldots, a_n) = A(a_n^{\prime}) + j \cdot \sum_{i = 1}^{2N} \h_i$, (again, for the same $j$ as in (ii));
				\end{itemize} 
			\end{enumerate}
		\end{lemma}
		
		Lemma~\ref{alg4} follows from Theorem 4.1 of~\cite{Kh2024}. It also follows from this theorem that each nonzero higher operation has the following output:
		\begin{equation}\label{alg5}
			\mu_n(a_1, \ldots, a_n) = \begin{cases}V_0^j &\text{ if } (a_1, \ldots, a_n) \text{ is centered;} \\ a_1^{\prime} V_0^j &\text{ if } (a_1, \ldots, a_n) \text{ is left-extended;} \\ V_0^j a_n^{\prime} & \text{ if } (a_1, \ldots, a_n) \text{ is right-extended;} \end{cases}
		\end{equation}
		with $n = j (2N -2) + 2$, and notation for left- and right-extended sequences of elements as in the statement of Lemma~\ref{alg4}. 
		
		Now, we have:
		\begin{proposition}\label{alg5}
			\emph{(Theorem 4.3 of~\cite{Kh2024})} With operations as specified in Lemma~\ref{alg4}, $\A$ satisfies $\A_{\infty}$-relations, and is a bona fide $\A_{\infty}$-algebra.
		\end{proposition}

		Finally, note that all non-zero operations on $\A$ preserve the Alexander grading, and transform the Maslov grading according to
		\begin{equation}\label{alg7}
			m(\mu_n(a_1, \ldots a_n)) = \sum_{i = 1}^n m(a_i) + n - 2.
		\end{equation}

		\subsection{The $\A_{\infty}$-algebra $\B$}\label{betaAlg}

		As in the case of $\A$, we first define an associative $R$-algebra $\B_0$, and then define $\B$ as a deformation of $\B_0$. Define $\B_0$ to be an associative algebra with generators
		\begin{itemize}
			\item $I_j$, $j = 1, \ldots, N$;
			
			\item $\rho_j$, $j = 1, \ldots, N$;

			\item $\sigma_j$, $j = 1, \ldots, N$;
		\end{itemize}
		$\B_0$ is also equipped with Maslov and Alexander gradings $m: \B_0 \to \Z$ and $A: \B_0 \to H_0(S^1 \smallsetminus \{p_1, \ldots, p_{2N}\}; \Z)$ (where we use the same generators $\h_1, \ldots \h_{2N}$ for $H_0(S^1 \smallsetminus \{p_1, \ldots, p_{2N}\};\Z)$ as in Section~\ref{alphaAlg}). Now define
		\begin{align*}
			m(\rho_j) = m(\sigma_j) &= -1 \text{ for each } 1 \leq j \leq N; \\ 
			A(\rho_j) &= \h_{2j - 1} \text{ for each } 1 \leq j \leq N; \\
			A(\sigma_j) &= \h_{2j} \text{ for each } 1 \leq j \leq N
		\end{align*}
		Gradings for the $V_i$ are the same as in $\A_0$. Define multiplication ($\mu_2$) on $\B_0$ as 
		\begin{align*}
			I_i I_j &= 0 \text{ unless } i = j \\
			I_i \rho_j = \rho_j I_i &= \begin{cases} \rho_i & \text{ if } i = j \\ 0 & \text{ otherwise} \end{cases} \\
			I_i \sigma_j &= \begin{cases} \sigma_i & \text{ if } i = j \\ 0 &\text{ otherwise} \end{cases} \\
			\sigma_j I_i &= \begin{cases} \sigma_i & \text{ if } j= i - 1 \\ 0 &\text{ otherwise} \end{cases} \\
			\rho_i \rho_j &= 0 \text{ for each } 1 \leq i, j \leq N; \\
			\sigma_i \sigma_j &= 0 \text{ for each } 1 \leq i \leq N; \\
			\rho_i \sigma_j &= 0 \text{ unless } i = j \\ 
			\sigma_i \rho_j &= 0 \text{ unless } i = j - 1
		\end{align*}
		The $I_i$ are again called \emph{idempotents}. As in $\A_0$, we define the \emph{initial idempotent} $I_{j_1}$ and \emph{final idempotent} $I_{j_2}$ of a non-zero element $b \in \B_0$ as the unique idempotents such that $I_{j_1} b$ and $b I_{j_2}$ are non-zero. Again, we often refer to $1 \leq j_1, j_2 \leq N$ as the initial and final idempotents of $b$, respectively.
		
		As in $\A_0$, we define
		\[
			U_0 = \sum_{i = 1}^N \sigma_{i + N} \rho_{i + N} \cdots \sigma_{i + 1} \rho_{i + 1} \sigma_i \rho_i + \sum_{i = 1}^N \rho_i \sigma_{i + N} \rho_{i + N} \cdots \sigma_{i + 1} \rho_{i + 1} \sigma_i
		\]
		Also, we define $\B_+ = \B_0 \smallsetminus \{I_i\}_{i = 1}^N$, and the augmentation map $\epsilon: \B_0 \to R$ to have $\ker \epsilon = \B_+$. We also define the \emph{length} of an element $b \in \B$, and the Alexander grading $A(b_1, \ldots b_n)$ of a sequence $b_1, \ldots, b_n \in \B$, analogously to~\eqref{alg61} and~\eqref{alg62}, respectively.

	Next, we define $\B$ as a deformation of $\B_0$. Unlike the definition of $\B$ in Section 5.2 of~\cite{Kh2024}, we only need a single family of higher multiplications, namely:
	\begin{equation}\label{alg9}
		\mu_N(\sigma_{i+N - 1}, \ldots, \sigma_{i + 1}, \sigma_i) = V_{N + 1},
	\end{equation}
	for all $1 \leq i \leq N$, with all indices counted $\emm N$. Now,
	\begin{proposition}\label{alg8}
		With operations as above, $\B$ satisfies the $\A_{\infty}$-relations and is a bona fide $\A_{\infty}$-algebra.
	\end{proposition}
	
	The deformed algebra $\B$ is not exactly the same one constructed in~\cite{Kh2024} -- it has many fewer $\A_{\infty}$-operations -- and therefore the proof of Proposition~\ref{alg8} is trivial in this case.

	\section{Uniqueness of $\A$ and $\B$}\label{deformations}
	
	The goal of this section is to verify the following proposition. 
	\begin{proposition}\label{uni1} 
		\emph{(Uniqueness of deformations)}
		\begin{enumerate}[label = (\alph*)]
			\item Up to isomorphism, there is a unique $\A_{\infty}$-deformation of $\A_0$ over $R = \F_2[V_0, V_{N + 1}]$ satisfying:
			\begin{itemize}
				\item The deformation is bigraded by $\ell$ and $m$, and for each generating element $U_i, s_i \in \A_0$, the gradings agree with the definitions given in Section~\ref{alphaAlg};

				\item The deformation has no non-zero operations $\mu_n$ for $2 < n < 2N$, and has 
				\begin{equation}\label{uni36}
					\mu_{2N}(U_1, s_1, \ldots, U_N, s_N) = V_0,
				\end{equation}
				and likewise for each cyclic permutation of the set $\{U_1, s_2,\ldots, U_N, s_N\}$, and no other non-zero $\mu_{2N}$s. 
			\end{itemize}
			
			\item Up to isomorphism, there is a unique $\A_{\infty}$-deformation of $\B_0$ over $R$ satifying: 
			\begin{itemize}
				\item The deformation is bigraded by $m$ and $A$, and for each generating element $\rho_i, \sigma_i \in \B_0$, the gradings agree with the definitions given in Section~\ref{betaAlg};

				\item The deformation has no non-zero operations $\mu_n$ for $2 < n < N$, and has
				\begin{equation}\label{uni37}
					\mu_N(\sigma_N, \sigma_{N -1}, \ldots, \sigma_1)  = V_{N + 1}
				\end{equation}
				and likewise for each cyclic permutation of the set $\{\sigma_N, \ldots, \sigma_1\}$, and no other non-zero $\mu_N$s.
			\end{itemize}
		\end{enumerate}
	\end{proposition}

	\subsection{Computing the cobar algebras of $\A_0$ and $\B_0$}\label{HochComp}
	
	In order to prove Proposition~\ref{uni1}, we are going to compute the cohomology of the subcomplex $C^*$ of the Hochschild cochain complex defined in the second half of Section~\ref{hoch}. First, note that $\A_0$ and $\B_0$ are both associative algebras over the ground ring $R = \F_2[V_0, V_{N + 1}]$, equipped with a Maslov grading and each with a set of idempotent elements. Therefore all of the results from Section~\ref{hoch} apply. By Proposition~\ref{uni29}, the first step is to compute $\tCob(\A_0)$ and $\tCob(\B_0)$.

	\begin{lemma}\label{uni2}
		Given $\A_0$ and $\B_0$ as in Sections~\ref{alphaAlg} and~\ref{betaAlg}, we have
		\begin{equation}\label{uni3}
			\tCob(\A_0) \simeq \B_0,
		\end{equation}
		and
		\begin{equation}\label{uni4}
			\tCob(\B_0) \simeq \A_0.
		\end{equation}
	\end{lemma}
	
	\begin{proof}
		Start with the case of $\A_0$, that is,~\eqref{uni3}. We are going to manually define a Maslov-graded quasi-isomorphism
		\[
			\varphi: \tCob(\A_0) \to \B_0
		\]
		Recall first that the Maslov grading on $\tCob(\A_0)$ is determined by~\eqref{uni32} for each element $a \in \A_0$. Since $m(a) = 0$ for each such $a$, this means that $m(a^*) = -1$ for each $a$.
		
		Now define $\varphi$ as follows:
		\begin{align*}
			\varphi(U_i^*) &= \rho_i \text{ for each } 1 \leq i \leq N; \\
			\varphi(s_i^*) &= \sigma_i \text{ for each } 1 \leq i \leq N; \\
			\varphi(a^*) &= 0 \text{ if } a \in \A \text{ with }\ell (a)  > 1;
		\end{align*}
		
		\begin{remark}\label{uni19}
			\emph{Note that although our usual ground ring is $R = \F_2[V_0, V_{N + 1}]$, we are currently working \emph{only} over $\F_2$. We will use Proposition~\ref{uni29} to deal with the case of $V_0 \in \A_0$.}
		\end{remark}
		
		We extend $\varphi$ in the obvious way, i.e. 
		\begin{equation}\label{uni20}
			\varphi(a_1^* \otimes \cdots \otimes  a_n^*) = \varphi(a_n) \cdot \varphi(a_{n - 1}) \cdots \varphi(a_1)
		\end{equation}
		(with the order reversed because multiplication in $\B_0$ is right-to-left rather than left-to-right). Now define
		\[
			\psi: \B_0 \to \tCob (\A_0)
		\]
		as 
		\begin{align*}
			\psi(\rho_i) &= U_i; \\
			\psi(\sigma_i) &= s_i;
		\end{align*}
		extended multiplicatively as 
		\begin{equation}\label{uni21}
			\psi(b_n \cdots b_1) = \psi(b_1) \otimes \cdots \otimes \psi(b_n)
		\end{equation}
		where we are using the unique factorization such that $b_i$ is a length-one element of $(\B_0)_+$ (the complement of the set of idempotents in $\B_0$). 
		
		Clearly, $\varphi \circ \psi = \id_{\B_0}$. We now need to define a chain homotopy operator $H: \tCob(\A_0) \to \tCob(\A_0)$ such that
		\begin{equation}\label{uni22} 
			\delta^{\tCob} \circ H + H \circ \delta^{\tCob} = \id_{\tCob(\A_0)} + \psi \circ \varphi.
		\end{equation}
		Note that $\psi \circ \varphi = \id_{\tCob(\A_0)}$ whenever on the complement of $\ker \varphi$. This means we just need to construct $H$ satisfying~\eqref{uni22} for $\xi \in \tCob(\A_0)$ with $\varphi(\xi) = 0$. To do this, note first that any element of $\imm \psi \subseteq  \tCob(\A_0)$ can be written uniquely in the form $a_1^* \otimes \cdots \otimes a_n^*$ for length-one elements $a_i \in \A_0$ such that $\varphi(a_n) \cdots \varphi(a_1) \in \B_0$ is non-zero. By the multiplication rules for $\B_0$, such a string will always be of the form
		\begin{enumerate}[label = (\Alph*)]
			\item $U_i^* \otimes s_i^* \otimes \cdots \otimes s_{i + n - 1}^* \otimes  U_{i + n}^*$;
			
			\item $U_i^* \otimes s_i^* \otimes \cdots  \otimes  U_{i + n}^* \otimes s_{i + n}^*$;
			
			\item $s_{i - 1}^* \otimes U_i^* \otimes \cdots \otimes s_{i + n - 1}^* \otimes  U_{i + n}^*$;
			
			\item $s_{i - 1}^* \otimes U_i^* \otimes \cdots \otimes  U_{i + n}^* \otimes s_{i + n}^*$;
		\end{enumerate}
		where all indices are counted $\emm N$. This means that each element of $\tCob(\A_0)$ can be written uniquely in the form 
		\begin{equation}\label{uni33}
			\xi = a_1^* \otimes \cdots \otimes a_n^* \otimes \tilde{a}_1^* \otimes \cdots \otimes \tilde{a}_m^*
		\end{equation}
		where each $a_1^* \otimes \cdots \otimes a_n^*$ a string of one of the forms (A)-(D), and $\tilde{a}_1^*$ is the first functional in the sequence such that either $\ell(\tilde{a}_1) > 1$ or $\ell(\tilde{a}_1) = 1$ and $ \varphi(\tilde{a}_1)\varphi(a_n) = 0$. A string $\xi$ of the form in~\eqref{uni33} has $\varphi(\xi) = 0$ if and only if $m > 0$. 
		
		Define
		\[
			H(a_1^* \otimes \cdots \otimes a_n^* \otimes \tilde{a}_1^* \otimes \cdots \otimes \tilde{a}_m^*) = \begin{cases} a_1^* \otimes \cdots \otimes a_{n - 1}^* \otimes (a_n \tilde{a}_1)^* \otimes \: \tilde{a}_2^* \otimes \cdots \otimes \tilde{a}_m^* & \text{ if } n > 0 \\0 &\text{ otherwise} \end{cases}
		\]
		Note first that for $\xi$ as in~\eqref{uni33}, $H(\xi) = 0$ whenever $\varphi \neq 0$. It just remains to verify~\eqref{uni22} for each $\xi \in \tCob(\A_0)$, written as in~\eqref{uni33}. The cases are as follows:
		\begin{enumerate}[label = {Case \arabic*:}]
			\item $m = 0$, i.e. $\xi = a_1^* \otimes \cdots \otimes a_n^* \in \imm \varphi$. Then $H(\xi) = 0$ and $\delta^{\tCob} \xi = 0$, and 
			\[
				(\psi \circ \varphi) (\xi) = \xi,
			\]
			as in~\eqref{uni22}.
			
			\item $n = 0$, or $m, n > 0$ with $a_n\tilde{a}_1 = 0$. In either of these cases, $H(\xi) = 0$, and $\varphi(\xi) = 0$. We want to show that $H \circ \delta^{\tCob} = \id_{\tCob(\A_0)}$. 
			
			If $n = 0$, i.e. $\xi = \tilde{a}_1^* \otimes \cdots \otimes \tilde{a}_m^*$ then $\ell(\tilde{a}_1) > 1$, and writing $\tilde{a}_1$ as a product $a_{11} \otimes a_{1\ell}$ of basic elements, we get that 
			\[
				\delta^{\tCob}(\xi) = \sum_{1 \leq i \leq (\ell - 1)} (a_{11}\cdots a_{1i})^* \otimes (a_{1(i + 1)} \cdots a_{\ell})^* \otimes \tilde{a}_2^* \otimes \cdots \otimes \tilde{a}_m^* + A,
			\]
			where $A$ is a sum of terms with first tensor factor $\tilde{a}_1^*$, and hence, with $H(A) = 0$. This means that
			\begin{align*}
				H \delta^{\tCob}(\xi) &= \sum_{1 \leq i \leq (\ell - 1)} H((a_{11}\cdots a_{1i})^* \otimes (a_{1(i + 1)} \cdots a_{\ell})^* \otimes \tilde{a}_2^* \otimes \cdots \otimes \tilde{a}_m^*) \\
				&= H(a_{11}^* \otimes (a_{12} \cdots a_{1\ell})^* \otimes \tilde{a}_2^* \otimes \cdots \otimes \tilde{a}_m^*) \\
				&= \xi.
			\end{align*}
			That is,~\eqref{uni22} holds.
			
			If $m, n > 0$ with $a_n \tilde{a}_1 = 0$. First note that we cannot have $\ell(\tilde{a}_1) = 0$. Because of the way we chose $\tilde{a}_1$, we would then have $\varphi(\tilde{a}_1) \varphi(a_n) = 0$; but because $\xi \in C^*$, we know the final idempotent of $a_n$ matches the initial idempotent of $\tilde{a}_1$; this means that either $\varphi(\tilde{a}_1) \varphi(a_n) = 0$ or $a_n \tilde{a}_1 = 0$, but not both. We can therefore assume in this case that $\ell(\tilde{a}_1) > 1$. Then the argument is analogous to the one for $n = 0$ above.
			
			\item $m, n > 0$ with $a_n \tilde{a}_1 \neq 0$. Then because we are working with $\xi \in C^*$, we get $\varphi(\tilde{a}_1) \varphi(a_n) = 0$, so $\varphi(\xi) = 0$, and we want to verify that
			\begin{equation}\label{uni34}
				(H \circ \delta^{\tCob} + \delta^{\tCob} \circ H)(\xi) = \xi.
			\end{equation}
			For this, just note that 
			\begin{align*}
				\delta^{\tCob} \xi &= a_1^* \otimes \cdots \otimes a_n^* \otimes \left( \sum_{1 \leq j \leq \ell - 1}  (a_{11} \cdots a_{1j})^* \otimes (a_{1(j + 1)} \otimes a_{1r})^*\right) \otimes \tilde{a}_2^* \otimes \cdots \otimes \tilde{a}_m^* \\
				&\qquad \qquad + a_1^* \otimes \cdots \otimes a_n^* \otimes \tilde{a}_1^* \otimes \delta^{\tCob}(\tilde{a}_2^* \otimes \cdots \otimes \tilde{a}_m^*)
			\end{align*}
			while
			\[
				H \xi = a_1^* \otimes \cdots \otimes a_{n - 1}^* \otimes (a_n \tilde{a}_1)^* \otimes \tilde{a}_2^* \otimes \cdots \otimes \tilde{a}_m^*,
			\]
			so that 
			\begin{align*}
				H \delta^{\tCob} \xi &= a_1^* \otimes \cdots \otimes a_{n - 1}^*  \otimes \left( \sum_{1 \leq j \leq \ell - 1}  (a_n a_{11} \cdots a_{1j})^* \otimes (a_{1(j + 1)} \otimes a_{1r})^*\right) \otimes \tilde{a}_2^* \otimes \cdots \otimes \tilde{a}_m^* \\
				&\qquad \qquad + a_1^* \otimes \cdots \otimes a_{n - 1}^* \otimes (a_n \tilde{a}_1)^* \otimes \delta^{\tCob}(\tilde{a}_2^* \otimes \cdots \otimes \tilde{a}_m^*)
			\end{align*}
			while $\delta^{\tCob} H \xi$ is $\xi$ plus the two terms from $H \delta^{\tCob} \xi$, above. This verifies~\eqref{uni34}
		\end{enumerate}

		The proof that $\tCob \B_0 \simeq \A_0$ is a straightforward modification of the proof that $\tCob \A_0 \simeq \B_0$. above.
	\end{proof}
	
	\begin{remark}\label{uni18}
		\emph{The identification of the cobar is related to Koszul duality, but in this case, we can prove it directly, without referring to the duality claim. A result analogous to Lemma~\ref{uni2} becomes more difficult when we are considering the cobars of bona fide $\A_{\infty}$-algebras }.
	\end{remark}

		Additionally, a key fact to note is that:
		\begin{lemma}\label{uni6}
			Both $\A_0$ and $\B_0$ are filtered by the length grading $\ell$, and $F_n \A_0$, $F_n \B_0$ are finite dimensional for each $n$.
		\end{lemma}		 

		Having noted all this, we can apply Lemma~\ref{hoch3} to ascertain the necessary facts about the Hochschild cohomologies of $\A_0$ and $\B_0$. Write $C^{*,*}(\A_0)$ and $C^{*,*}(\B_0)$ for the subcomplexes of $\HC^{*,*}(\A_0)$ and $\HC^{*,*}(\B_0)$, respectively, which were discussed in Section~\ref{hoch}, and write $H^{*,*}(\A_0), H^{*,*}(\B_0)$ for the (bigraded) cohomology groups corresponding to these chain complexes.

	\subsection{Computing the Hochschild cohomology}\label{uniPunch}

		The last step before we can apply Corollary~\ref{uni27} to is to compute $H^{*, -1}$ an $H^{*,-2}$ of $\A_0$ and $\B_0$. In order to make this computation, we are going to combine~\eqref{uni30} from Proposition~\ref{uni30} with~\eqref{uni3} and~\eqref{uni4} from Lemma~\ref{uni2} in order to compute $H^{*,*}(\A_0)$ using the chain complex $\A_0[V_0, \ldots, V_{N + 1}] \otimes \B_0$, and compute $H^{*,*}(\B_0)$ using the chain complex $\B_0[V_0, \ldots, V_{N + 1}] \otimes \A_0$. The result is as follows. 
		\begin{proposition}\label{uni5}
			With notation as above,
			\begin{equation}\label{uni7}
				H^{n, -2}(\A_0) = \begin{cases} R & n = 2N \\ 0 & \text{ otherwise}\end{cases}
			\end{equation}
			and
			\begin{equation}\label{uni15}
				H^{n, -1}(\A_0) = 0 \text{ for all } n \geq 2.
			\end{equation}
			For $\B_0$, we have
			\begin{equation}\label{uni16}
				H^{n, -2}(\B_0) = \begin{cases} R & n = N \\ 0 & \text{ otherwise} \end{cases}
			\end{equation}
			and
			\begin{equation}\label{uni17}
				H^{n, -1}(\B_0) = 0 \text{ for all } n.
			\end{equation}
		\end{proposition}
		
		\begin{proof}
			We compute $H^{n,j}(\A_0)$ first. Recall that $n$ is the length-grading, and $j$ is the Maslov grading. Also note that we are now working over the full ground ring $R = \F_2[V_0, V_{N + 1}]$, with the understanding that for $\rho \in \A_0$, $1 \leq i \leq N + 1$, $V_i \rho = \rho V_i = 0$. For each $\rho \in \A_0$, we can write $\rho$ uniquely in the form 
			\begin{equation}\label{uni9}
				\rho = \rho_0 \cdot V_0^j
			\end{equation}
			where $\rho_0$ contains no factors of $V_0$. Now, for $\rho \in \A_0$,
			\begin{equation}\label{uni8}
				m(\rho) = \begin{cases} j(2N - 2) & \text{ if } j > 0 \text{ in the representation from~\eqref{uni9};} \\ 0 & \text{ otherwise.} \end{cases} 
			\end{equation}
			
			Now, for $\rho = \rho_0 V_0^j$, any non-zero element of $\rho \otimes \tau \in \A_0[V_0, \ldots, V_{N + 1}] \otimes \B$ will be of the form
			\[
				\rho_0 V_0^j \otimes U_0^j \tau_0,
			\]
			where $\tau_0 \in \B_0$ is the element dual to $\rho_0 \in \A_0$ according to the procedure from the proof of Lemma~\ref{uni2}. This means that 
			\begin{align*}
				m(\rho_0 V_0^j \otimes U_0^j\tau_0) &= 2jN + m(U_0^j) + m(\tau_0) \\
				&= j(2N - 2) - 2jN - \ell (\tau_0) \\
				&= - 2j - \ell (\tau_0) \\
				&= -2j - \ell(\rho_0).
			\end{align*}
			Note that we have just shown that there are no allowable $\rho \otimes \tau \in \A_0[V_0,\ldots, V_{N + 1}]$ with $m(\rho \otimes \tau) = -1$ and length $> 2$. This verifies~\eqref{uni15}. 
			
			This calculation also means that $m(\rho) = -2$ if and only if $j = 1$ and $\rho_0$ is trivial, i.e. $\rho = V_0$, or $j = 0$ and $\ell(\rho_0) = 2$. But we are only interested in $\HH_{\Z}^{n,*}(\A_0)$ for $n > 2$, so this latter case is not relevant here. In order to verify~\eqref{uni7}, we just need to show that 
			\begin{enumerate}[label = (\roman*)]
				\item $\del (V_0 \otimes U_0) = 0$;
				
				\item $V_0 \otimes U_0$ is not a coboundary. 
			\end{enumerate}
			For (i), recall that
			\[
				U_0 = \sum_{i = 1}^N \sigma_{i + N - 1} \rho_{i + N - 1} \cdots \sigma_i \rho_i + \sum_{i = 1}^N \rho_i \sigma_{i + N - 1} \cdots \rho_{i + 1} \sigma_i
			\]
			so using multiplication rules on $\B_0$, we get
			\begin{align*}
				\del (V_0 \otimes U_0) &= \sum_{i = 1}^N ( U_i V_0 \otimes \rho_i \sigma_{i + N - 1} \cdots \rho_{i + 1} \sigma_i \cdot \rho_i + V_0 U_i \otimes \rho_i \cdot \sigma_{i + N - 1} \rho_{i + N - 1} \cdots \sigma_i \rho_i) \\
				&\qquad + \sum_{i = 1}^N (s_{i - 1} V_0 \otimes \sigma_{i + N - 1} \rho_{i + N - 1} \cdots \sigma_i \rho_i \sigma_{i - 1} + V_0 s_i \otimes \sigma_i \cdot \rho_i \sigma_{i + N - 1} \cdots \rho_{i + 1} \sigma_i) \\
				&= 0
			\end{align*}
			For (ii), note that any $\rho \otimes \tau \in \A_0[V_0, \ldots, V_{N + 1}] \otimes \B_0$ with $\del (\rho \otimes \tau) = V_0 \otimes U_0$ would have $m(\rho \otimes \tau) = - 3$ and $\ell(\rho) = \ell(\tau) = 2N - 1$. But there are no such elements. 
			
			For $H^*(\B_0)$, we  note that for each $\rho \in \B_0$, we can write $\rho$ uniquely in the form
			\[
				\rho  = \rho_0 V_1^{j_1} \cdots V_{N + 1}^{N + 1}.			\]
			where each $j_i$ is a nonnegative integer and $\rho_0$ contains no factors of any $V_i$. If $j_i > 0$ for $1 \leq i \leq N$, then there is not $\tau \in \A$ such that $\rho \otimes \tau$ is a non-zero element of $\B_0[V_0, \ldots, V_{N + 1}] \otimes \A_0$. Hence, we restrict ourselves to the case of $\rho = \rho_0 V_{N + 1}^j$, for $j$ a non-negative integer and $\rho_0$ not containing any factors of $V_i$. Then if $\tau \in \A_0$ is such that $\rho \otimes \tau$ is a nonzero element of $\B_0[V_0, \ldots, V_{N + 1}] \otimes \A_0$, then we have
			\begin{equation}\label{uni35}
				m(\rho \otimes \tau) = - 2 j  - \ell(\rho_0),
			\end{equation}
			Note that $\ell(\rho) = jN + \ell(\rho_0)$. Hence, for $\ell(\rho)  > 2$, we never have never have $m(\rho \otimes \tau) = -1$. This verifies~\eqref{uni7} Also,~\eqref{uni35} means that for $\ell(\rho) > 2$, $m(\rho \otimes \tau) = - 2$ implies that $\ell(\rho_0) = 0$ (i.e. $\rho_0$ is trivial) and $j = 1$ ? that is, $\rho = V_{N + 1}$. Now, the only nonzero element of $\B_0[V_0, \ldots, V_{N + 1}] \otimes \A_0$ with $\B_0[V_0, \ldots, V_{N + 1}]$ component $V_{N + 1}$ is $V_{N + 1} \otimes U_{N + 1}$, which has $\ell(V_{N + 1} \otimes U_{N + 1}) = N$. This verifies~\eqref{uni6}.
		\end{proof}

		We are now ready to prove Proposition~\ref{uni1}, that is, that $\A$ and $\B$ are the unique deformations of $\A_0$ and $\B_0$ with the given $\mu_{2N}$ and $\mu_N$ operations, respectively.  

		\begin{proof}[Proof of Proposition~\ref{uni1}]
			The proof is a direct application of Corollary~\ref{uni27} combined with the calculations from Proposition~\ref{uni5}. 

			The only slightly touchy thing is that the result of Proposition~\ref{uni5} provides a single Hochschild cochain corresponding to the $\mu_{2N}$ for $\A$, and a single Hochschild cochain corresponding to the $\mu_N$ for $\B$. This would seem to imply that there is a \emph{single} non-zero $\mu_{2N}$ operation on $\A$ and a single $\mu_N$ operation on $\B$, rather than the $2N$ non-zero $\mu_{2N}$s for $\A$ stipulated by~\eqref{uni36} and the $N$ non-zero $\mu_N$s stipulated by~\eqref{uni37}. However, each of the operations from~\eqref{uni36} is in fact different components (in separate Alexander gradings) of the single Hochschild cocycle corresponding to the $\mu_{2N}$; indeed, looking back at the explicit computation of the Hochschild cohomology from the proof of Proposition~\ref{uni5}, it is clear that the cocycle corresponding to the $\mu_{2N}$ has one component in each Alexander grading, just like the operations stipulated by~\eqref{uni36}. The case of the $\mu_N$ for $\B$ is analogous.
		\end{proof}

	\section{Proof of Theorem~\ref{fukaya}}\label{proof}

		The goal of this section is to apply Proposition~\ref{uni1} in order to verify Theorem~\ref{fukaya}. For both $\A$ and $\B$, the proof has the same three steps. Using $\A$ as an example, the first step of the proof is to identify the generators of $\A_+$ (the augmentation ideal in the underlying associative algebra $\A_0$) with the elements of $\End_{\W(M)}(\alpha_1 \oplus \cdots \oplus \alpha_N)$. The second step is to show that for each $x \in \End_{\W(M)}(\alpha_1 \oplus \cdots \oplus \alpha_N)$ corresopnding to an element of $a \in \A_+$, the $\Z$-component of $\gr(x)$ equals $m(a)$. Finally, we do a series of model calculations which show that $\End_{\W_{\alpha}}(\alpha_1 \oplus \cdots \oplus \alpha_N)$ satisfies the conditions of Proposition~\ref{uni1}(a); that is, that
		\begin{itemize}
			\item Multiplication preserves the $\Z$-component of $\gr$;

			\item There are no non-zero multiplication operations of order $2 < n < 2N$, and there are  $\mu_{2N}$ in each Alexander grading, corresponding to the cyclic permutations of~\eqref{uni36}. 
		\end{itemize}
		
		\begin{wrapfigure}{r}{3cm} \vspace{-20pt}
			\includegraphics[width = 3cm]{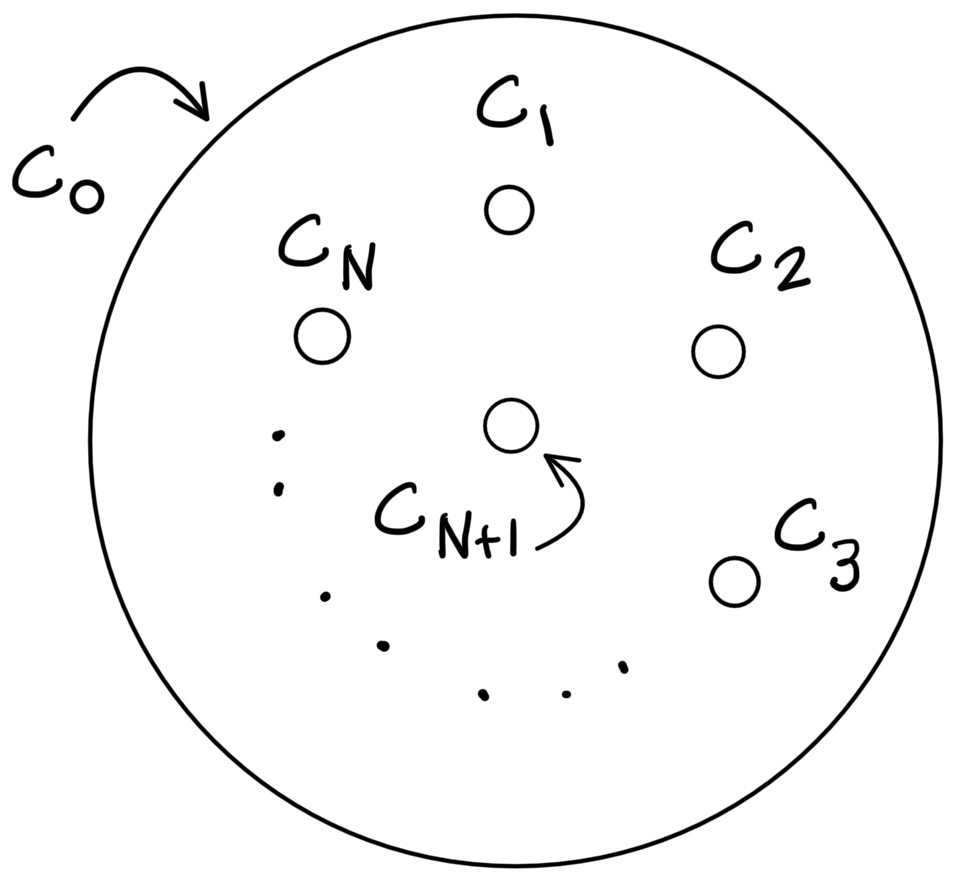}
			
			\caption{}\label{proof5} \vspace{-20pt}
		\end{wrapfigure}
		
		Let $C_0, \ldots, C_{N + 1}$ denote the boundary circles of the star diagram, labelled as in Figure~\ref{proof5}. Write
		\begin{align*}
			\tilde{\A} &= \End_{\W(M)}(\alpha_1 \oplus \cdots \oplus \alpha_N) \\
			\tilde{\B} &= \End_{\W(M)}(\beta_1 \oplus \cdots \oplus \beta_N)
		\end{align*}
		Let $\tilde{m}$ denote the $\Z$-valued gradings induced on both $\tilde{A}$ and $\tilde{\B}$ as the projection of $\gr$ to the $\Z$-coordinate. The $\alpha$- and $\beta$-arcs are defined as in Figure~\ref{proof80}, for arbitrary $N > 2$. (Figure~\ref{proof80} has $N = 3$.)
		
		\begin{wrapfigure}{r}{3.5cm}
			\includegraphics[width = 3.5cm]{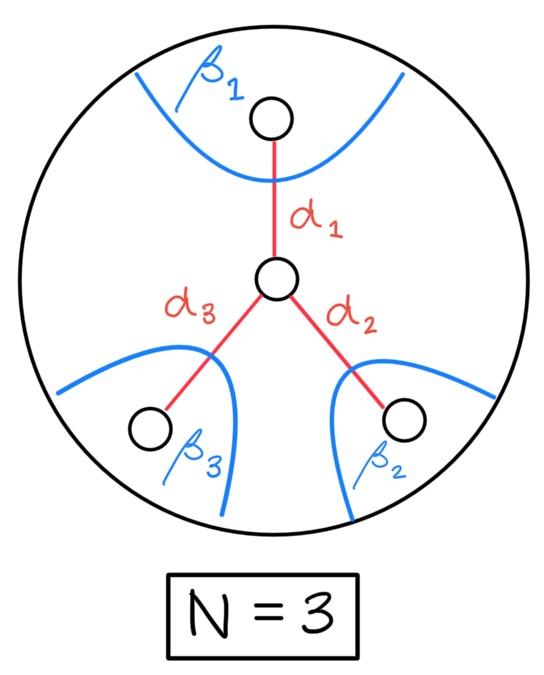}
			
			\caption{}\label{proof80}
		\end{wrapfigure}

		To prove~\eqref{fukaya1} from Theorem~\ref{fukaya}, we first identify the $I_i, U_i$ and $s_i, (1 \leq i \leq N)$, with elements of  $\tilde{\A}$. For each $1 \leq i \leq N$, define $\tilde{\alpha}_i$ to be a slight perturbation of $\alpha_i$ such that $\tilde{\alpha}_i \cap C_i$ and $\tilde{\alpha}_i \cap C_{N + 1}$ are slightly clockwise of $\alpha_i \cap C_i$ and $\alpha_i \cap C_{N + 1}$, respectively. Define $\alpha_i^j = \varphi^j(\tilde{\alpha}_i)$. Then 
		\begin{equation}\label{proof1}
			q_{ii}^0 := \alpha_i \cap \tilde{\alpha}_i,
		\end{equation}
		is a single point on $\alpha_i$, and
		\begin{equation}\label{proof2}
			\alpha_i \cap \alpha_i^1 = \{q_{ii}^n, r_{ii}^n\}_{n = 1}^{\infty},
		\end{equation}
		\begin{wrapfigure}{r}{5cm}\vspace{-20pt}
			\includegraphics[width = 6cm]{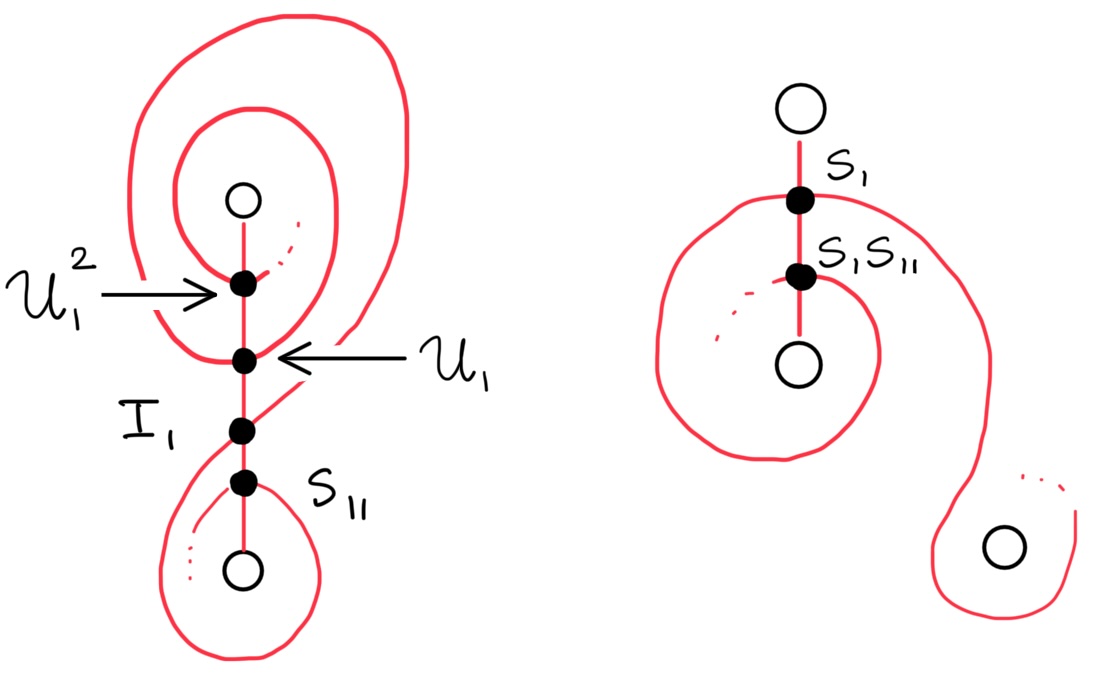}
		
		\caption{}\label{proof4}
		\end{wrapfigure}
		where the $q_{ii}^0$ are a sequence of points radiating inward from $q_{ii}^0$ along $\alpha_i$, towards $C_{N + 1}$, and $r_{ii}^n$ are a sequence of points radiating outward from $q_{ii}^0$ along $\alpha_i$ towards $C_i$. These points are as on the left of Figure~\ref{proof4}. Likewise, for $i \neq j$, 
		\begin{equation}\label{proof3}
			\alpha_i \cap \alpha_j^1 = \{q_{ij}^n\}_{n = 1}^{\infty}
		\end{equation}
		is a sequence of points radiating inward from $q_{ii}^0$ along $\alpha_i$, towards $C_{N + 1}$. These points are as on the right of Figure~\ref{proof4}.

		Note also that up to Hamiltonian isotopy, each of these can be viewed as the corresponding intersection point in $\alpha_i^r \cap \alpha_j^{r + 1}$ for any $r \in \N$. Note also that by Proposition~\ref{fukDef8}(b), the gradings of these two points are the same, so it does not matter which point we consider.

		The identification between generators of $\tilde{\A}$ and elements of $\A_0$ is as follows:
		\begin{align*}
			I_i &\leftrightarrow q_{ii}^0; \quad U_i^n \leftrightarrow r_{ii}^n \\
			s_{ii}^n &\leftrightarrow q_{ii}^n \quad s_{ij} s_{jj}^n \leftrightarrow q_{ij}^{n + 1} \text{ for } i \neq j;
		\end{align*} 
		For the rest of the section, we will abuse notation by writing the corresponding element of $\A_0$ instead of the $q_{ij}^k$ and $r_{ii}^n$. 
		
		Define a $\Z_{\geq 0}$-valued \emph{length grading} $\ell$ on $\tilde{\A}$ according to the rules:
		\begin{enumerate}[label = (\roman*)]
			\item $\ell(I_i) = \ell(U_i) = \ell(s_i) = 1$ for each $i$;
			
			\item $\ell(V_0) = 2N$
			
			\item $\ell(\mu_n(a_1, \ldots, a_n) ) = \sum_{i = 1}^n \ell(a_i)$.
		\end{enumerate}
		
		The next step is to verify that the non-zero $\mu_2$s on $\tilde{\A}$ are the same as in $\A_0$, and that the $\mu_2$ on $\tilde{\A}$ preserves $\tilde{m}$.

		\begin{lemma}\label{end1}
			\begin{enumerate}[label = (\alph*)]
				\item The identifications above define a bijection between $\tilde{\A}$ and $\A_0$; 
				
				\item Each $U_i, s_i$ has trivial $\tilde{m}$-grading;

				\item The non-zero $\mu_2$ operations on $\tilde{\A}$ are in one-to-one correspondence with the non-zero $\mu_2$ operations on $\A$.

				\item These non-trivial $\mu_2$ operations preserve $\tilde{m}$;
			\end{enumerate}
		\end{lemma}

		\begin{proof}
			(a) follows from Figure~\ref{proof4}. (d) follows directly from Proposition~\ref{fukDef8}(a), and (b) then follows from Figures~\ref{proof12} and~\ref{proof72}. There are two things to note about Figures~\ref{proof12} and~\ref{proof72}. First is that we are grossly abusing notation by denoting by $\tilde{b}_{\bullet}$ the respective projections in the original manifold $M$ (along with shadows of the coresponding element of $\lLag(2) = S^1$. Second is that the copy of  $S^1$ which forms $\lLag(2)$ is actually the quotient of $S^1$ under the action of $\{\pm 1\}$; therefore, we should actually be considering the \emph{quotients} of the curves pictured on the right of Figures~\ref{proof12} and~\ref{proof72} under the $\{\pm 1\}$-action. However, the curves pictured below convey the same information, and are easier to draw.
			
		\begin{figure}[H]
			\includegraphics[width = 10cm]{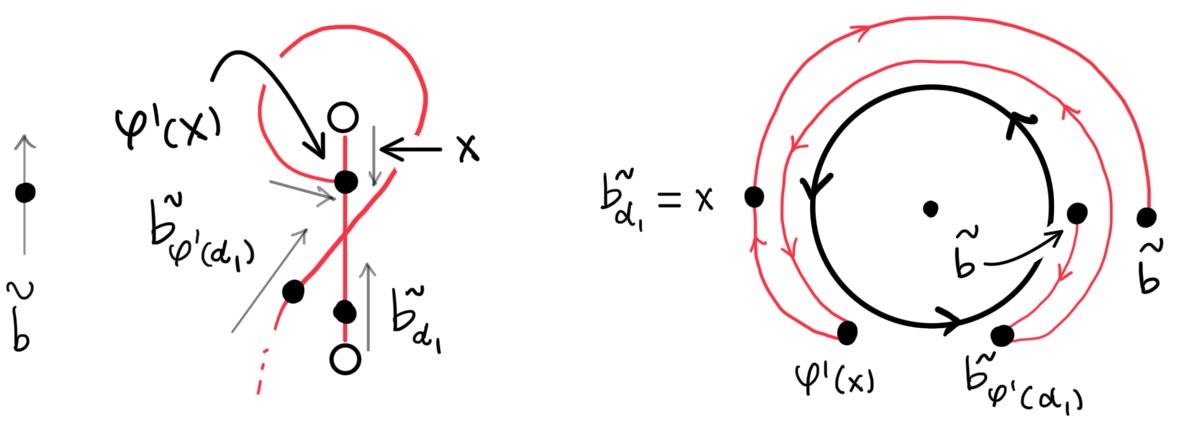}
			
			\caption{}\label{proof12} 
		\end{figure}
		
		\begin{figure}[H]
			\includegraphics[width = 10cm]{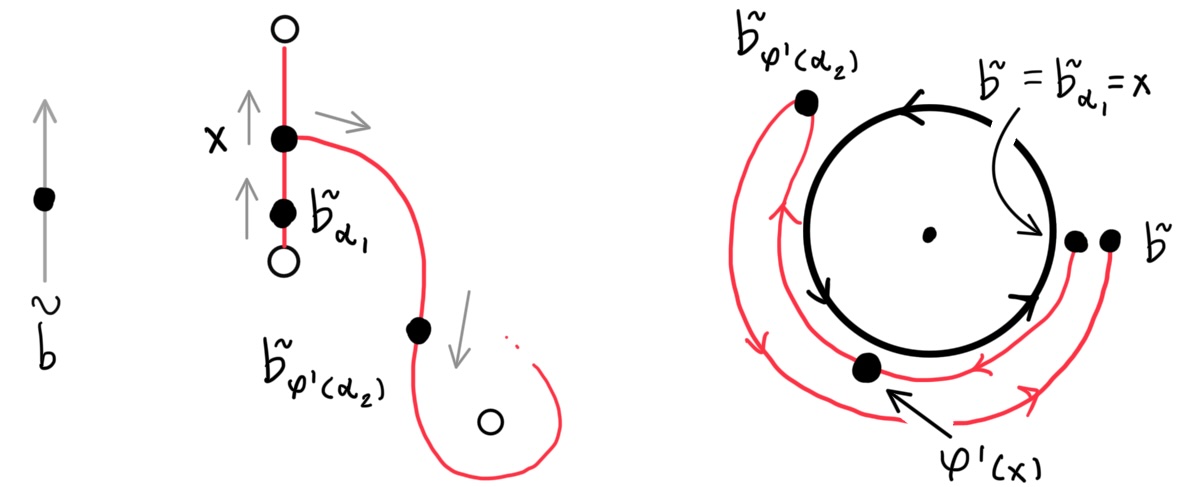}
			
			\caption{}\label{proof72}
		\end{figure}

			For (c), we need to verify first that all the multiplications from $\A_0$ correspond to $\mu_2$ operations on $\tilde{\A}$, and then the converse. For the first direction, the fact that multiplication on $\tilde{\B}$ is associative implies that we only need to verify that the multiplications match up in the case of the basic multiplications. That is, we need only show that for each $1 \leq i \leq N$, there exist unique (correctly oriented) triangles with
				\begin{itemize}
					\item Vertices $U_i, U_i$, and $U_i^2$, and edges on $\alpha_i, \alpha_i^1,$ and $\alpha^2$;
					\item $s_i, s_{i + 1}$ and $s_{i (i + 2)}$ and edges on $\alpha_i$, $\alpha_{i + 1}^1$, and $\alpha_{i + 2}^2$;
				\end{itemize}
				up to the identifications above, and with all indices counted $\emm N$. These triangles are pictured (for $N = 3, i = 1$) in Figure~\ref{proof13}.
				
		\begin{figure}[H]
			\includegraphics[width = 6.5cm]{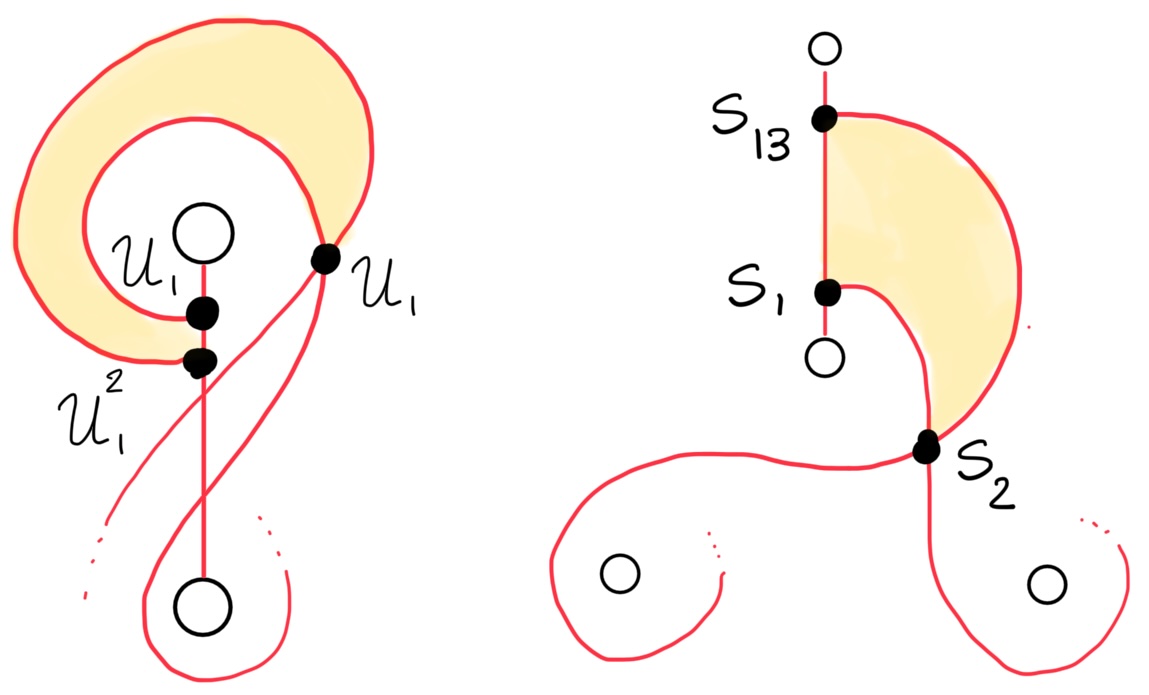}
			
			\caption{}\label{proof13}
		\end{figure}
		
		 Hence, every non-zero $\mu_2$ in $\A_0$ corresponds to a non-zero multiplication on $\tilde{\A}$. The last step is to show that these are the \emph{only} non-zero $\mu_2$'s on $\tilde{\A}$. To do this, note that if $a b \neq 0$ in $\tilde{\A}$, then the final idempotent of $a$ matches the initial idempotent of $b$. The options (for basic multiplication) which are not already covered are $s_{i - 1} U_{i}$ and $U_i s_i$ (with all indices counted $\emm N$). But these cannot be connected by any Maslov index one, correctly oriented triangle with the correct edges (which are spun correctly), as illustrated in Figure~\ref{proof17} (for the generic case $i = 1$). 
		 \end{proof}
		 \begin{figure}[H]
			\includegraphics[width = 6.5cm]{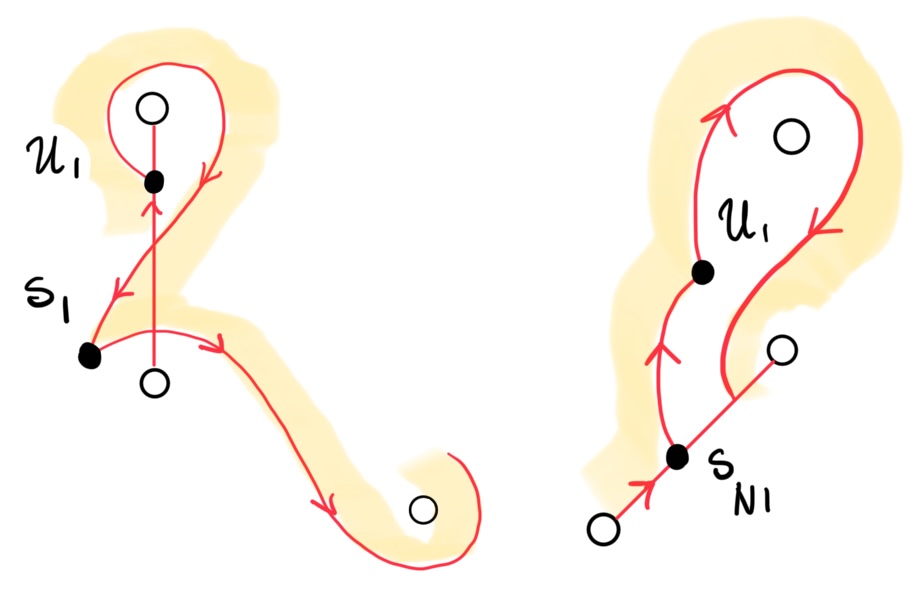}
			
			\caption{}\label{proof17}
		\end{figure}

		The next step is to verify

		\begin{lemma}\label{end2}
			\begin{enumerate}[label = (\alph*)]
				\item There are no non-trivial operations $\mu_n$ on $\tilde{\A}$, with $2 < n < 2N$;
				
				\item There are exactly $2N$ distinct $\mu_{2N}$ operations on $\tilde{\A}$, which are precisely those from~\eqref{uni36} of Proposition~\ref{uni1};
			\end{enumerate}
		\end{lemma}

		\begin{figure}[H]
				\includegraphics[width = 12cm]{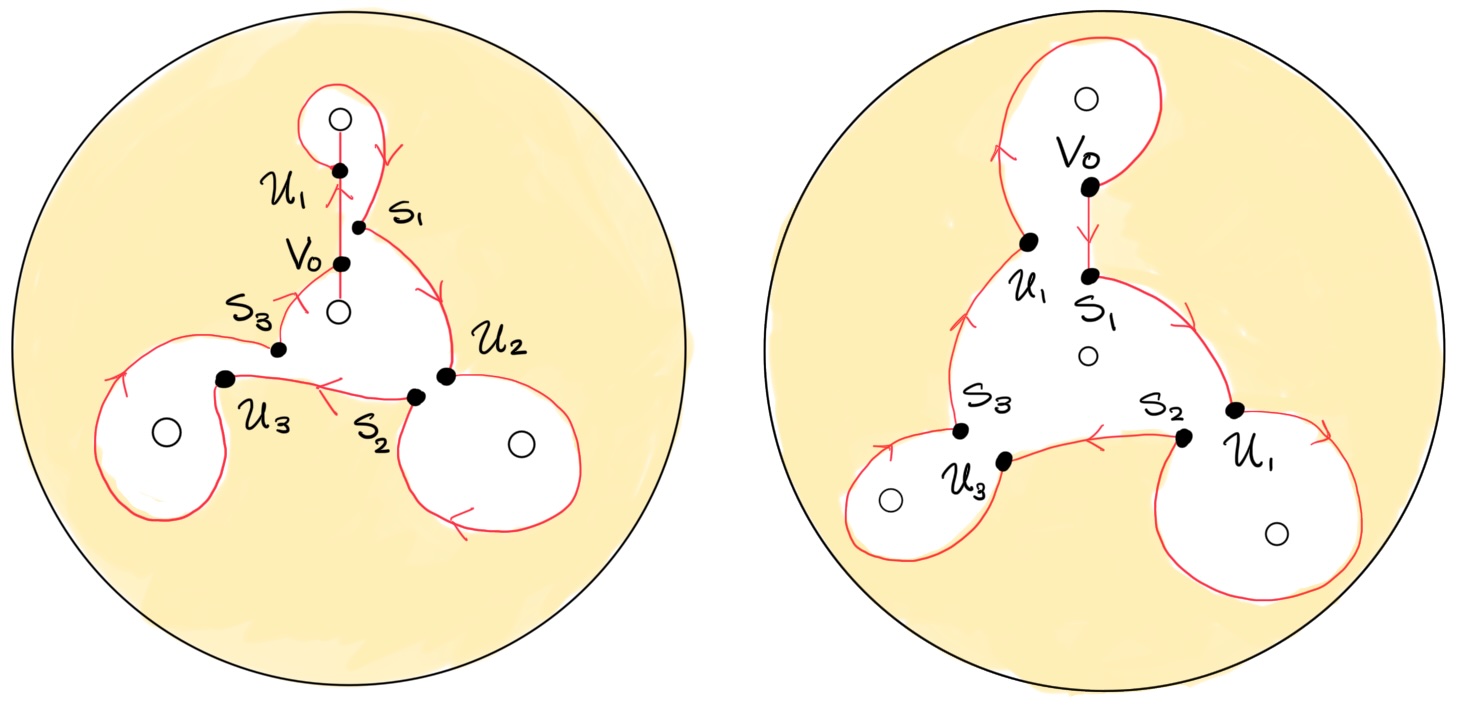}
			
				\caption{}\label{proof6}
		\end{figure}
		
		\begin{proof}
			(a) follows for grading reasons as in the verification of the analogous fact for $\A$. To reiterate, we know from Proposition~\ref{end1}(a) that $\tilde{m}(a) = 0$ for each $a \in \tilde{A}$; hence, by Proposition~\ref{fukDef8}(a), $\tilde{m}$ and any operation on $\tilde{A}$ must satisfy
			\[
				\tilde{m}(\mu_n(a_1, \ldots, a_n)) = \sum_{i = 1}^n \tilde{m}(a_i) + n - 2 = n - 2.
			\]
			However, no element of $\tilde{\A}$ has grading $n - 2$ for $2 < n < 2N - 2$. 
			
			(b) follows directly from the two pictures in Figure~\ref{proof6}, up to cyclic permutation of labels. After noting that these are in fact the unique $(2N + 1)$-gons with boundary on the desired $\alpha$-arcs the proof is complete. 
		\end{proof}

		For $\B_0$, the steps are analogous. For each $1 \leq i \leq N$, let $\tilde{\beta}_i$ be a perturbation of $\beta_i$ such that the points of $\tilde{\beta}_i \cap C_0$ are slightly clockwise of the corresponding point of $\beta_i \cap C_0$. Define $\beta_i^j = \varphi^j(\tilde{\beta}_i)$, and define
		\begin{equation}\label{proof7}
			\tilde{q}_{ii}^0 := \beta_i \cap \tilde{\beta}_i,
		\end{equation}
		and
		\begin{equation}\label{proof8}
			\beta_i \cap \beta_i^1 = \{\tilde{q}_{ii}^n, \tilde{r}_{ii}^n\}_{n = 1}^{\infty},
		\end{equation}
		and for $i \neq j$,
		\begin{equation}\label{proof3}
			\beta_i \cap \beta_j^1 = \{\tilde{q}_{ij}^n\}_{n = 1}^{\infty}
		\end{equation}
		The identification between elements of $\tilde{\B}$ and elements of $\B_0$ is as follows
		\begin{align*}
			I_i &\leftrightarrow \tilde{q}_{ii}^0; \quad \rho_i^n \leftrightarrow \tilde{r}_{ii}^n \\
			s_{ii}^n &\leftrightarrow q_{ii}^n \quad s_{ij} s_{jj}^n \leftrightarrow q_{ij}^{n + 1} \text{ for } i \neq j;
		\end{align*} 
		These points are pictured (labelled by the corresponding element of $\B_0$) in Figure~\ref{proof9}.
		
		\begin{figure}[H]
			\includegraphics[width = 10cm]{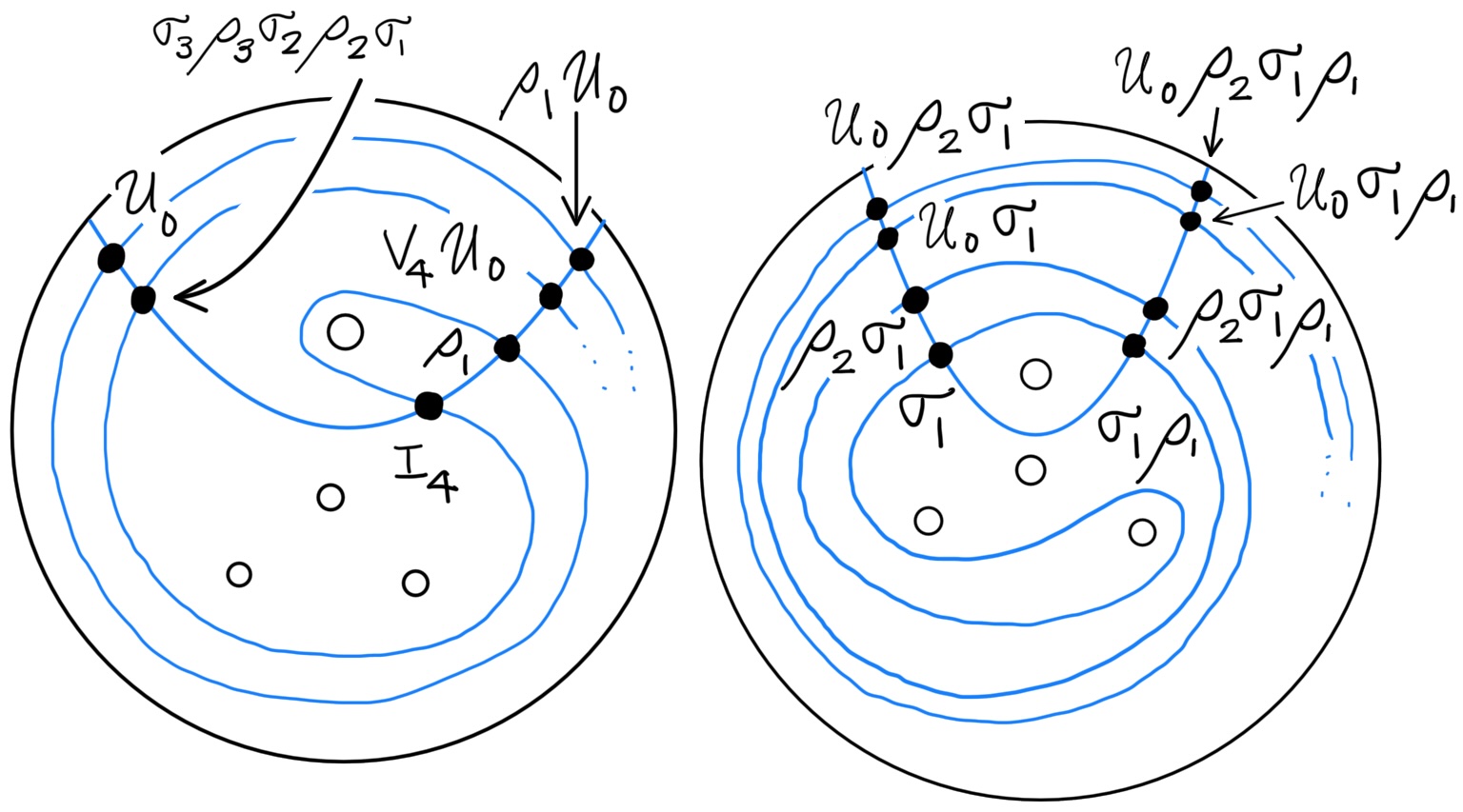}
			
			\caption{}\label{proof9}
		\end{figure}
		
		Define a $\Z_{\geq 0}$-valued length grading $\ell$ on $\tilde{\B}$ according to the same three rules as for $\tilde{\A}$, that is:
		\begin{enumerate}[label = (\roman*)]
			\item $\ell(I_i) = \ell(\rho_i) = \ell(\sigma_i) = 1$ for each $i$;
			
			\item $\ell(V_{N + 1}) = N$;
			
			\item We also require that higher operations satisfy
			\[
				\ell(\mu_n(b_1, \ldots, b_n) ) = \sum_{i = 1}^n \ell(b_i).
			\]
		\end{enumerate}

		\begin{lemma}\label{end3}
			\begin{enumerate}[label = (\alph*)]
				\item The identifications above define a bijection between $\tilde{\B}$ and $\B_0$; 
			
				\item Each $b = \rho_i, \sigma_i$ has $\tilde{m}(b) = -1$;

				\item The non-zero $\mu_2$ operations on $\tilde{\B}$ are in one-to-one correspondence with the non-zero $\mu_2$ operations on $\B$.

				\item These non-trivial $\mu_2$ operations preserve $\tilde{m}$ in the sense that 
				\[
					\tilde{m}(\mu_2(b_1, b_2)) = \tilde{m}(b_1) + \tilde{m}(b_2).
				\]
			\end{enumerate}
		\end{lemma}
		
		\begin{proof}
			(a) follows from Figure~\ref{proof9}. (d) again follows from Proposition~\ref{fukDef8}, and (b) follows from calculations analogous to the ones done in Figures~\ref{proof12} and~\ref{proof72} for $\tilde{\A}$.

			For (c), we want to show that non-zero $\mu_2$ on $\B$ implies non-zero $\mu_2$ on $\tilde{\B}$, and then the reverse. In more detail, the forward direction says that if $b_1, b_2 \in \B$ so that $b_2 b_1 \neq 0$, and we let $i_1, i_0$ denote the initial and final idempotents of $b_1$, respectively, and let $i_2$ denote and the initial idempotent of $b_2$, then we want to find a unique, Maslov-grading 1, triangle with vertices $b_2, b_1,$ and $b_2b_1$, and with boundary along $\beta_{i_0}, \beta_{i_1}^1$, and $\beta_{i_2}^2$. In fact, because multiplication on $\tilde{\B}$ is associative, we only need to verify it for the basic elements of $\B_+$, that is, the $\sigma_i$ and $\rho_i$. For the basic elements 
			\begin{itemize}
				\item $b_1 = \rho_1, b_2 = \sigma_1$;
				
				\item $b_1 = \sigma_1, b_2 = \rho_2$;
			\end{itemize}
			these triangles are pictured in Figure~\ref{proof14}. For all other multipliable pairs of basic elements in $\B_+$, the triangles are analogous. Now, we just need to verify that given a Maslov index 1 triangle with vertices $b_2, b_1,$ and $b_3$, and boundary along $\beta_{i_0}, \beta_{i_1}^1$, and $\beta_{i_2}^2$, $b_2 b_1 \neq 0$ in $\B$ and is equal to $b_3$. The argument that these are the \emph{only} non-zero $\mu_2$'s in $\tilde{\B}$ is analogous to the case of $\tilde{\A}$ (but with order of multiplication reversed).
		\end{proof}
		
		\begin{figure}[H]
			\includegraphics[width = 8cm]{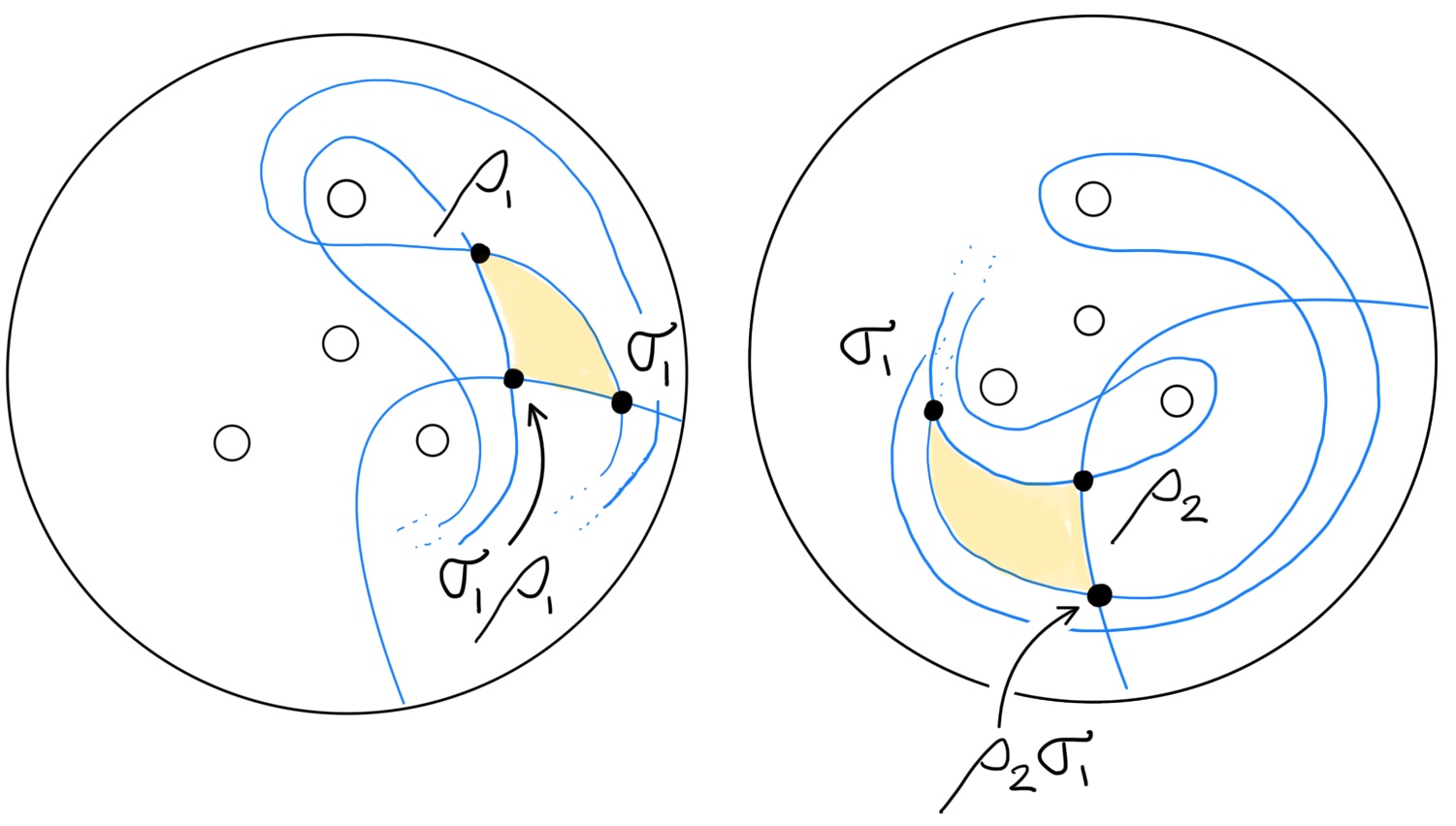}
			
			\caption{}\label{proof14} 
		\end{figure}

		\begin{lemma}\label{end4}
			\begin{enumerate}[label = (\alph*)]
				\item There are no non-trivial operations $\mu_n$ on $\tilde{\A}$, with $2 < n < N$;
				
				\item There are exactly $N$ distinct non-extended $\mu_{N}$ operations on $\tilde{\A}$, which are precisely those from~\eqref{uni37} of Proposition~\ref{uni1};
			\end{enumerate}
		\end{lemma}
		
		\begin{proof}
			(a) follows for grading reasons; indeed, consider $\mu_n(b_1, \ldots, b_n) = b_0$ with $2 < n < N$ and $b_i \in \tilde{\B}$ for each $i$. Then
			\begin{equation}\label{proof10}
				\tilde{m}(b_0) = \sum_{i = 1}^n \tilde{m}(b_i) + n- 2 = n - 2 - \sum_{i = 1}^n \ell(b_i).
			\end{equation}
			Moreover, $\ell(b_0) = \sum_{i = 1}^n \ell(b_i)$. If $\ell(b_0) < N$, then $\tilde{m}(b_0) = - \ell(b_0)$, so that~\eqref{proof10} implies $n = 2$, which we are not considering; we can therefore exclude this case. If $\ell(b_0) \geq N$, then $b_0$ contains some $j$ factors of $V_{N + 1}$, which contribute $-2j$ to $\tilde{m}(b_0)$, and $+j$ to $\ell(b_0)$.  
			
			Write $\tau_1, \ldots \tau_k \in \B_+$ for the basic factors of $b_1, \ldots b_n$. Arguing using the fact that $\sum_{i = 1}^N \ell(b_i) = \ell(b_0)$, we see that if we temporarily ignore the $\tau_i$ which correspond to the factors of $b_0$, we are left with $jN$ basic factors. These $\tau_i$ contribute $-jN$ to the sum $\sum_{i = 1}^n \tilde{m}(b_i)$. The equality of~\eqref{proof10} now reduces to:
			\begin{equation}\label{proof11}
				-2j = -jN + n - 2. 
			\end{equation}
			We are assuming $n < N$, which means we have
			\[
				j = \frac{n - 2}{N - 2} < 1,
			\]
			which contradicts $j \in \N$. Hence, (a) follows.
			
			For (b),~\eqref{proof11} likewise shows that if $n = N$, then we need $j = 1$. Because we are only interested in the generating  Hochschild cocycle -- that is, the unextended $\mu_N$ operation -- we want to consider the case when 
			\begin{equation}\label{proof75}
				\ell(b_0) = \sum_{i = 1}^N \ell(b_i) = N,
			\end{equation}
			and show that the $\mu_N$ must be of the form
			\begin{equation}\label{proof16}
				\mu_{N} (  \sigma_{i + N - 1}, \sigma_{i + N - 2}, \ldots, \sigma_{i + 1}, \sigma_i) = V_{N + 1},
			\end{equation}
			where the indices are counted $\emm N$ as usual.  In general, because we are only considering the case where~\eqref{proof75} holds, and all basic elements of $\tilde{\B}$ have length $1$, we have $b_0 = V_{N + 1}$. It remains to show that 
			\[
				(b_1, \ldots, b_N) = (  \sigma_{i + N - 1}, \sigma_{i + N - 2}, \ldots, \sigma_{i + 1}, \sigma_i),
			\]
			for some $1 \leq i \leq N$. Arguing using~\eqref{proof75} again, we know that each $b_i$ is a length $1$ element, and is therefore one of the $\rho_j$ or $\sigma_j$. Because we are only looking at sequences $(b_1, \ldots, b_n)$ so that the initial idempotent of $b_i$ is the final idempotent of $b_{i - 1}$ for each $i$, it will suffice to show that each $b_i$ must be \emph{some} $\sigma_j$.
			
			We do this by considering the canonical $\mu_N$ operation and its corresponding $(N + 1)$-gon. That there exists such a (unique) non-trivial $\mu_N$ with the desired input and output follows directly from Figure~\ref{proof15} (and the analogous figures for cyclic permutations of $\{1, \ldots N\}$ and arbitrary $N$).
		\begin{figure}[H]
			\includegraphics[width = 5cm]{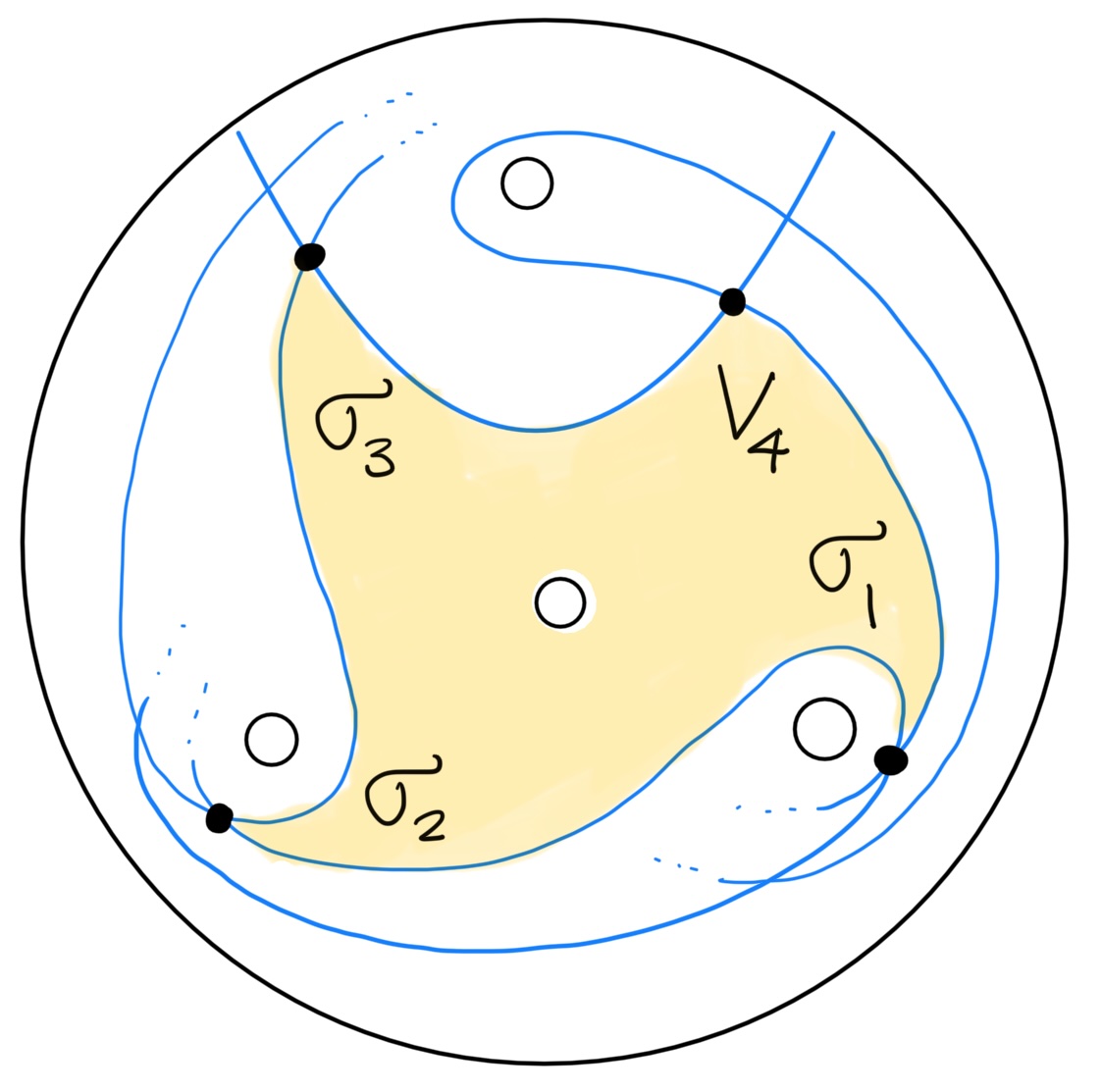}
			
			\caption{}\label{proof15}
		\end{figure}
			\noindent Moreover, the calculations in Figures~\ref{proof73} and~\ref{proof74} show that the $F_{N + 1}$-components of $\gr(\rho_i)$ and $\gr(\sigma_i)$ are $\overline{i}$ and $1$, respectively.
			
			\begin{figure}[H]
				\includegraphics[width = 9cm]{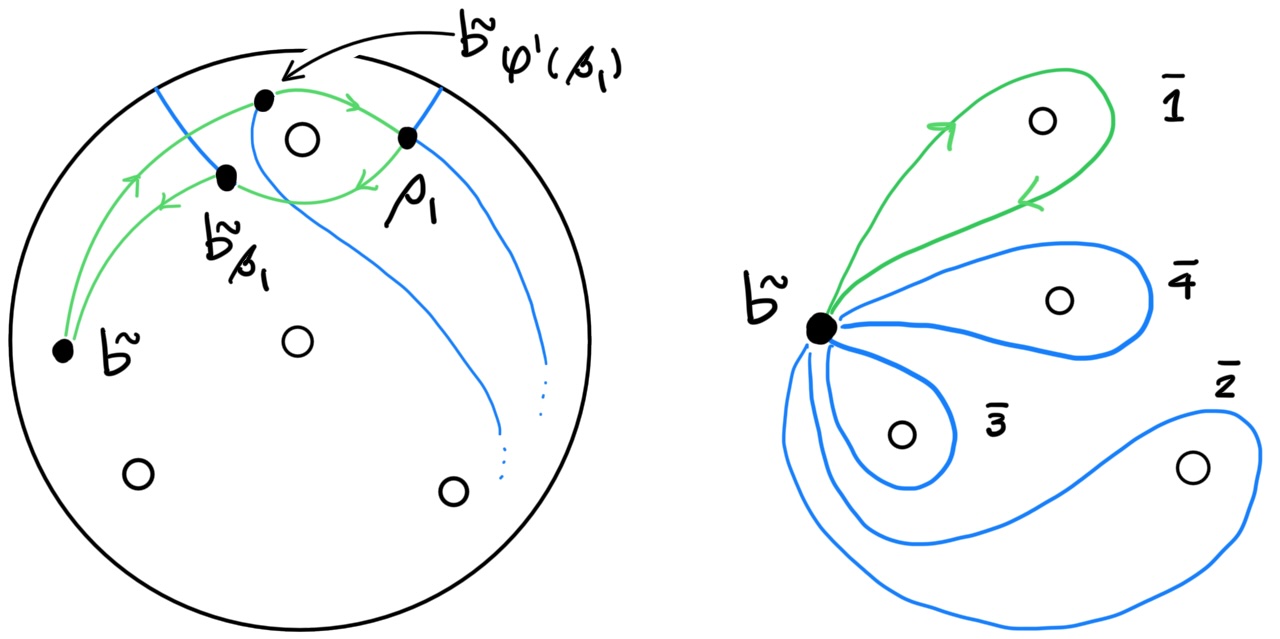}
			
				\caption{}\label{proof73}
			\end{figure}
		
			\begin{figure}[H]
				\includegraphics[width = 9cm]{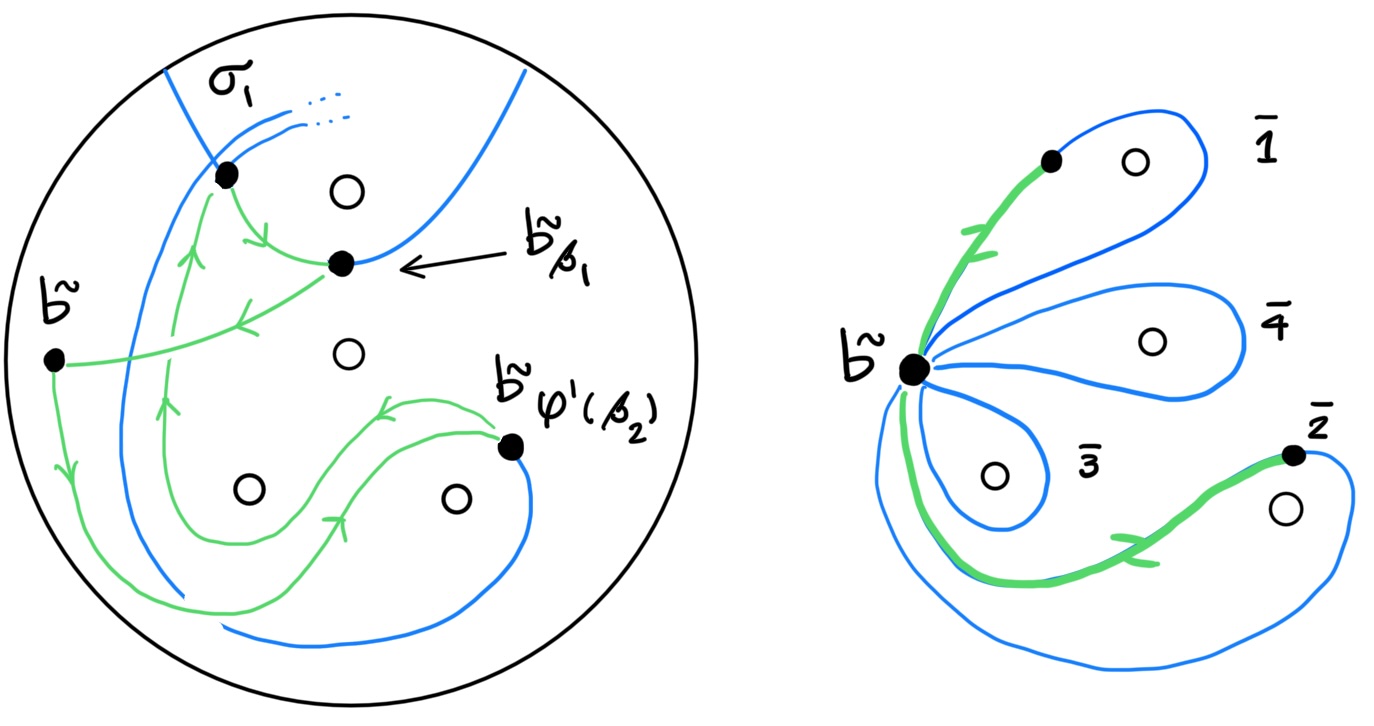}
			
				\caption{}\label{proof74}
			\end{figure}
			
			\noindent Combined with Figure~\ref{proof15} and Proposition~\ref{fukDef8}(a), Figure~\ref{proof74} shows that the $F_{N + 1}$-component of $\gr(V_{N + 1})$ is equal to $1$. Therefore, by Figures~\ref{proof73} and~\ref{proof74} again, combined with Proposition~\ref{fukDef8}(a) as before, each $b_i$ must be one of the $\sigma_i$. Given the discussion of idempotents above, this suffices to verify that the operation is of the form in~\eqref{proof16}.
		\end{proof}

		 Lemmas~\ref{end1} through~\ref{end4} confirm that $\tilde{\A}$ and $\tilde{\B}$ satisfy the conditions of Proposition~\ref{uni1}. Applying Proposition~\ref{uni1}, we can now conclude that Theorem~\ref{fukaya} holds.

\end{document}